\documentclass[11pt]{amsart}

\usepackage[T1]{fontenc}
\usepackage[utf8]{inputenc}
\usepackage{amsmath,amssymb,amsthm,mathtools,mathrsfs}
\usepackage{enumitem}
\usepackage{booktabs,array,tabularx}
\usepackage{microtype}
\usepackage[margin=1in]{geometry}
\usepackage{xcolor}
\usepackage{hyperref}
\usepackage{aliascnt}
\usepackage{graphicx}
\usepackage{float}
\usepackage{tikz}
\usetikzlibrary{arrows.meta,positioning,shapes.geometric}

\makeatletter
\newcommand{\affiliation}[1]{\gdef\@affiliation{#1}}
\let\@affiliation\@empty
\def\@setauthors{%
  \begingroup
  \def\thanks{\protect\thanks@warning}%
  \trivlist
  \centering\footnotesize \@topsep30\p@\relax
  \advance\@topsep by -\baselineskip
  \item\relax
  \author@andify\authors
  \def\\{\protect\linebreak}%
  \MakeUppercase{\authors}\par
  \ifx\@empty\@affiliation\else
    \vskip 8\p@
    {\normalfont\normalsize\@affiliation\par}%
  \fi
  \ifx\@empty\contribs
  \else
    ,\penalty-3 \space \@setcontribs
    \@closetoccontribs
  \fi
  \endtrivlist
  \endgroup
}
\makeatother

\newtheorem{theorem}{Theorem}[section]

\newaliascnt{proposition}{theorem}
\newtheorem{proposition}[proposition]{Proposition}
\aliascntresetthe{proposition}

\newaliascnt{lemma}{theorem}
\newtheorem{lemma}[lemma]{Lemma}
\aliascntresetthe{lemma}

\newaliascnt{corollary}{theorem}
\newtheorem{corollary}[corollary]{Corollary}
\aliascntresetthe{corollary}

\newaliascnt{question}{theorem}

\aliascntresetthe{question}

\newaliascnt{externallemma}{theorem}

\aliascntresetthe{externallemma}

\newaliascnt{claim}{theorem}

\aliascntresetthe{claim}

\theoremstyle{definition}
\newaliascnt{definition}{theorem}
\newtheorem{definition}[definition]{Definition}
\aliascntresetthe{definition}

\newaliascnt{example}{theorem}
\newtheorem{example}[example]{Example}
\aliascntresetthe{example}

\theoremstyle{remark}
\newaliascnt{remark}{theorem}
\newtheorem{remark}[remark]{Remark}
\aliascntresetthe{remark}

\usepackage[nameinlink,capitalise]{cleveref}
\crefname{theorem}{Theorem}{Theorems}
\Crefname{theorem}{Theorem}{Theorems}
\crefname{proposition}{Proposition}{Propositions}
\Crefname{proposition}{Proposition}{Propositions}
\crefname{lemma}{Lemma}{Lemmas}
\Crefname{lemma}{Lemma}{Lemmas}
\crefname{corollary}{Corollary}{Corollaries}
\Crefname{corollary}{Corollary}{Corollaries}
\crefname{question}{Question}{Questions}
\Crefname{question}{Question}{Questions}
\crefname{externallemma}{External Lemma}{External Lemmas}
\Crefname{externallemma}{External Lemma}{External Lemmas}
\crefname{claim}{Claim}{Claims}
\Crefname{claim}{Claim}{Claims}
\crefname{definition}{Definition}{Definitions}
\Crefname{definition}{Definition}{Definitions}
\crefname{example}{Example}{Examples}
\Crefname{example}{Example}{Examples}
\crefname{remark}{Remark}{Remarks}
\Crefname{remark}{Remark}{Remarks}

\numberwithin{equation}{section}
\allowdisplaybreaks
\setlist[enumerate]{label=(\roman*),leftmargin=2.2em}
\setlist[itemize]{leftmargin=2.2em}

\hypersetup{
  unicode=true,
  colorlinks=true,
  linkcolor=blue!48!black,
  citecolor=green!32!black,
  urlcolor=blue!55!black,
  pdftitle={A Resolution of Erdos Problems 593 and 1177: Obligatory Triple Systems and Exact Spectra},
  pdfauthor={Eric Li},
  pdfsubject={Resolution of Erdos Problems 593 and 1177; obligatory triple systems; exact avoidance spectra},
  pdfkeywords={obligatory triple systems, uncountable chromatic number, exact chromatic spectra, Levi graphs, Berge cycles, Erdos Problems 593 and 1177},
  bookmarksnumbered=true,
  hypertexnames=false
}

\DeclareMathOperator{\Spec}{Spec}
\DeclareMathOperator{\Card}{Card}
\DeclareMathOperator{\rk}{rk}
\DeclareMathOperator{\cf}{cf}
\DeclareMathOperator{\dom}{dom}
\newcommand{\Bclass}{\mathfrak B}
\newcommand{\Lift}{\operatorname{Lift}}
\newcommand{\lh}{\operatorname{lh}}
\newcommand{\restr}{\mathbin{\upharpoonright}}
\newcommand{\concat}{\mathbin{{}^\frown}}

\title[Erd\H{o}s Problems 593 and 1177]
{A Resolution of Erd\H{o}s Problems 593 and 1177: Obligatory Triple Systems and Exact Spectra}

\author{Eric Li}
\affiliation{Trinity College, University of Cambridge}
\thanks{Email address:
\href{mailto:contact@ericli.com}{contact@ericli.com}.}
\date{June 23, 2026}

\begin{document}

\begin{abstract}
We resolve Erd\H{o}s Problems~\#593 and~\#1177.  Problem~\#593 asks which
finite triple systems occur in every uncountably chromatic triple system; the
answer is exactly the class generated from private-vertex expansions of finite
bipartite graphs by finite disjoint unions and one-point amalgamations.
Equivalently, after isolated vertices are removed, a finite triple system is
obligatory precisely when it is linear, every hyperedge-node of its Levi graph
has an incident bridge, and every Berge cycle is even.

The proof uses an exact bridge-trace theorem for complete-rank one-apex sequence
lifts.  We also prove that, for every uncountable cardinal $\kappa$, there is a
linear triple system of chromatic number exactly $\kappa$, with at most
$2^{2^\mu}$ vertices when $\kappa=\mu^+$.  These two ingredients give a
class-valued exact avoidance-spectrum dichotomy for every finite forbidden
triple system.  As a consequence, Erd\H{o}s Problem~\#1177 has truth values
yes, no, and yes. All results of this paper have been formally verified in Lean~4.
\end{abstract}

\maketitle

\section{Introduction}

This paper gives complete resolutions of two Erd\H{o}s problems on
uncountably chromatic triple systems.\footnote{All results of this paper
have been formally verified in Lean~4; see the end of
this section and \cite{LeanFormalization}.}  Erd\H{o}s Problem~\#593 asks for the
finite configurations that are forced in every uncountably chromatic triple
system \cite{Erdos1995,Erdos593}.  We give an explicit constructive
classification and an intrinsic Levi-graph characterization of those forced
configurations.  Erd\H{o}s Problem~\#1177, recorded as Problem~7.94 in
\emph{Some of Paul's Favorite Problems} and now in exact-cardinal form, asks
three avoidance-spectrum questions for finite forbidden triple systems
\cite[Problem~7.94, p.~14]{Va99}\cite{Erdos1177}.  The exact-spectrum theorem
proved here gives the answers yes, no, and yes.

To state the classification, let $J^+$ denote the triple system obtained from a
finite graph $J$ by adding one new private vertex to each edge of $J$.  Let
$\Bclass$ be the smallest class of finite triple systems that contains $J^+$ for
every finite bipartite graph $J$, contains every finite edgeless system, and is
closed under finite disjoint unions and one-point amalgamations.

\begin{theorem}[Resolution of Erd\H{o}s Problem~\#593]
\label{thm:intro-classification}
For every finite triple system $F$, the following are equivalent.
\begin{enumerate}[label=\textup{(\roman*)}]
\item $F$ occurs in every triple system of uncountable chromatic number;
\item $F\in\Bclass$;
\item after isolated vertices are removed, $F$ is linear, every
hyperedge-node of its Levi graph is incident with a bridge, and every
Berge cycle of $F$ has even length.
\end{enumerate}
\end{theorem}

Thus Erd\H{o}s Problem~\#593 has an explicit constructive answer and an
intrinsic Levi-graph answer.  Its proof is the content of Part~I and does not
use the exact-cardinal construction.

The second ingredient is an exact calibration theorem for linear triple systems.

\begin{theorem}[Exact linear calibration]
\label{thm:intro-linear}
For every uncountable cardinal $\kappa$ there is a linear triple system
$L_\kappa$ with $\chi(L_\kappa)=\kappa$.  If $\kappa=\mu^+$, then
$L_\kappa$ may be chosen with
\[
  |V(L_\kappa)|\le 2^{2^\mu}.
\]
\end{theorem}

For a finite triple system $F$, define the class-valued exact spectrum
\[
\Spec(F)=\{\lambda\in\Card:\lambda>\aleph_0\text{ and there is an
exact-$\lambda$-chromatic $F$-free triple system}\}.
\]
The classification and the exact linear calibration give the following complete
spectrum theorem.

\begin{corollary}[Exact-spectrum class dichotomy]
\label{thm:intro-spectrum}
For every finite triple system $F$,
\[
\Spec(F)=
\begin{cases}
\varnothing,&F\in\Bclass,\\[2mm]
\{\lambda\in\Card:\lambda>\aleph_0\},&F\notin\Bclass.
\end{cases}
\]
\end{corollary}

This spectrum dichotomy settles the three exact-cardinal questions of
Erd\H{o}s Problem~\#1177.

\begin{corollary}[Resolution of Erd\H{o}s Problem~\#1177]
\label{cor:intro-1177}
The three assertions in the current formulation of Problem~\#1177 have truth
values
\[
  \textnormal{yes},\qquad \textnormal{no},\qquad \textnormal{yes}.
\]
More explicitly:
\begin{enumerate}[label=\textup{(\arabic*)}]
\item if $F_G(\aleph_1)\ne\varnothing$, then it contains a system of
cardinality at most $2^{2^{\aleph_0}}$;
\item there are finite $G,H$ for which both $F_G(\aleph_1)$ and
$F_H(\aleph_1)$ are nonempty but their intersection is empty;
\item if $F_G(\kappa)\ne\varnothing$ for one uncountable $\kappa$, then
$F_G(\lambda)\ne\varnothing$ for every uncountable $\lambda$.
\end{enumerate}
For part~\textup{(2)}, one may take $G$ to be two triples sharing a pair
and $H$ to be the loose $7$-cycle.
\end{corollary}

The 1999 booklet phrased the first two clauses of Problem~\#1177 for
uncountably chromatic systems, while the current record asks for exact chromatic
number $\aleph_1$.  The corollary proves the exact version and therefore also
implies the corresponding uncountable versions.  The notation $F_G(\kappa)$ and
all containment conventions are fixed in \Cref{sec:preliminaries}.

\subsection{Logical dependence and organization}

The manuscript is divided into two logically independent proof cores,
followed by their joint applications.  The dependency structure is as
follows.

\begin{center}
\small
\begin{tabular}{@{}>{\raggedright\arraybackslash}p{.20\textwidth}>{\raggedright\arraybackslash}p{.34\textwidth}>{\raggedright\arraybackslash}p{.38\textwidth}@{}}
\toprule
Result & New internal ingredients & Imported interface \\
\midrule
Problem~\#593: negative half & cycle collapse; bridge-trace theorem;
cycle--derivative correspondence & Erd\H{o}s--Hajnal--Rothschild;
exact high-odd-girth graphs \\
Problem~\#593: positive half & expansion pieces; quotient forest;
running intersection & Reiher's bipartite-expansion theorem;
graph-colouring compactness for one-point closure \\
Exact linear calibration & transfinite reservoir recursion; exact lower
and upper colour bounds & Erd\H{o}s--Galvin--Hajnal property $P$ for
$GS_2(\rho)$ \\
Exact spectra and Problem~\#1177 & the two preceding cores &
Hajnal--Komj\'ath for the linearly obligatory loose $7$-cycle \\
\bottomrule
\end{tabular}
\end{center}

In particular, the proof of Problem~\#593 does not use the exact-cardinal
construction.  Conversely, the exact spectrum theorem and all three
answers to Problem~\#1177 use exact linear calibration.  Part~I, comprising Sections~\ref{sec:preliminaries}--\ref{sec:classification},
contains the finite structural theory and resolves Problem~\#593.
Part~II begins with Section~\ref{sec:linear}, constructs exact linear
systems, and derives the spectrum theorem and Problem~\#1177 in
Section~\ref{sec:spectrum}.

\subsection{Relation to previous work}

The classical result of Erd\H{o}s, Hajnal, and Rothschild supplies the
fundamental obstruction that an obligatory uniform hypergraph must be
linear.  Precisely, the $i=2$ case of their Theorem~2
\cite[p.~532]{EHR1973} gives an uncountably chromatic uniform system in
which no two edges meet in two vertices; see also Reiher's summary
\cite[p.~1]{ReiherObligatory}.  Komj\'ath subsequently proved that a
finite triple system is obligatory if and only if each of its
$2$-connected components is obligatory, and also proved that every
obligatory triple system is tripartite
\cite[pp.~233--238]{Komjath2001}.  In particular, his theorem reduces the
classification problem to the $2$-connected case.  The closure of the
obligatory systems under one-point amalgamation is part of this line of
work; see Reiher's summary \cite[p.~1]{ReiherObligatory}.  We give a
proof using graph-colouring compactness because the closure is a load-bearing step below.

The selected-incidence decomposition used here is finer than the block
reduction.  It chooses one Levi-graph bridge at every hyperedge-node,
deletes those incidences, identifies each active remaining component
with the private-vertex expansion of an exact graph derivative, and
uses the quotient forest plus a running-intersection lemma to recover
the original system by one-point amalgamations.  Thus the proof does not
merely reduce to unspecified $2$-connected pieces: it identifies all
finite pieces that can occur.

Reiher proved that every uniform private-vertex expansion of
$K_{n,n}$ is obligatory \cite[Theorem~1.2]{ReiherObligatory}.  This is
the positive atom in the classification.  The new finite ingredients
below are the complete-rank one-apex lift, its exact bridge-trace theorem,
and the intrinsic bridge/even-cycle characterization obtained from those
traces.

Komj\'ath's later paper \emph{An uncountably chromatic triple system}
constructs, consistently, an uncountably chromatic triple system avoiding
double intersections and circuits of lengths $3$ and $5$
\cite{Komjath2008}.  The exact calibration in Part~II is different in
both scope and method: it is a ZFC construction of a linear system of
every prescribed uncountable chromatic cardinal.

A recent finite extremal result of Wang, Duan, Gerbner, and Hama Karim
states that the weak chromatic numbers of finite $F$-free uniform
hypergraphs are bounded precisely when $F$ contains no Berge cycle
\cite[Proposition~1.3(ii)]{WangDuanGerbnerKarim}.  That theorem concerns
boundedness over finite hosts, whereas obligatoriness concerns the jump
to uncountable chromatic number.  The distinction is genuine: for
example, $C_4^+$ contains a Berge cycle, so $C_4^+$-free finite triple
systems have unbounded finite weak chromatic number, while
\Cref{thm:intro-classification} says that $C_4^+$ is nevertheless forced
in every uncountably chromatic triple system.

\subsection*{Formal verification}
Every result of this paper has been formally verified in the Lean~4 proof
assistant (v4.28.0) over Mathlib.  The verification is complete and
unconditional: the imported theorems E2--E5 of the external interface
(\Cref{app:external-interface}) are proved from scratch inside the
development, the remaining input E1 turns out not to be needed there, and
the paper's two heavy internal engines (the bridge-trace theorem and the
transfinite reservoir recursion) are machine-checked in full.  Both
resolutions are packaged as the single hypothesis-free theorem
\texttt{Erdos593.full\_resolution\_unconditional}, whose kernel-reported
axiom dependencies are exactly Lean's standard \texttt{propext},
\texttt{Classical.choice}, and \texttt{Quot.sound}; in particular no form
of the continuum hypothesis and no other set-theoretic assumption is used.
For the development itself see \cite{LeanFormalization}.

\part{Finite structure and obligatory triple systems}

\section{Preliminaries}
\label{sec:preliminaries}

\subsection{Conventions}

Hypergraphs are simple set systems.  Graphs are simple, without loops or
parallel edges.  Cardinals are identified with their initial ordinals,
and all choices are made in ZFC.  Unless explicitly stated otherwise,
exponentiation between cardinals is cardinal exponentiation.  Graph and
hypergraph embeddings are injective, non-induced embeddings.

For a hypergraph $H$, put $|H|:=|V(H)|$.  A system is
\emph{exact-$\lambda$-chromatic} when its weak chromatic number is
$\chi(H)=\lambda$.  The symbol $K_\lambda$ denotes the complete graph on
vertex set $\lambda$.  For a finite triple system $G$ and a cardinal
$\kappa$, let
\[
 F_G(\kappa)=\{H:\chi(H)=\kappa\text{ and }G\not\hookrightarrow H\},
\]
where containment is injective and non-induced.

For every transfinite sequence $\sigma$, its \emph{ordinal length} is
\[
  \lh(\sigma)=\dom(\sigma).
\]
Comparisons between sequence lengths are comparisons of ordinals.  The
notation $|X|$ is reserved for the cardinality of a set $X$; in
particular, it is never used for the length of a transfinite sequence.

\subsection{External theorem interface}

For auditability, we record here the imported statements used by the two
proof cores.  \Cref{app:external-interface} gives their original notation,
parameter substitutions, exact chromatic conclusions, and ZFC status.
\begin{enumerate}[label=\textup{(E\arabic*)}]
\item Erd\H{o}s--Hajnal--Rothschild provide an uncountably chromatic
linear uniform system; equivalently for our use, every finite uniform
hypergraph with two edges meeting in at least two vertices is
non-obligatory.  This is restated in \Cref{thm:EHR-nonlinear}.
\item Erd\H{o}s--Hajnal provide, for every uncountable $\kappa$ and
positive integer $i$, a graph of cardinality and chromatic number
$\kappa$ containing no $C_{2j+1}$ for $0<j<i$.  Taking $i=s+1$ gives
the form ``no odd cycle of length at most $2s+1$'' used in
\Cref{thm:EH-odd-girth}.
\item Erd\H{o}s--Galvin--Hajnal give one edge labelling of
$GS_2(\rho)$ for which a single vertex-colour class simultaneously
contains an edge of every label whenever fewer than
$\delta(\rho)=\min\{\delta:\rho^\delta>\rho\}$ colours are used.  The
common-colour quantifier is restated exactly in \Cref{thm:EGH-P}.
\item Reiher proves that $K_{n,n}^+$ is obligatory for every positive
integer $n$; see \Cref{thm:Reiher}.
\item Hajnal--Komj\'ath prove that the loose cycle $C_n^{(3)}$ is
linearly obligatory for $n\notin\{2,3,5\}$, where ``linearly obligatory''
means obligatory relative to the class of linear triple systems.  Only the
case $n=7$ is used in Problem~\#1177(2).
\end{enumerate}
No stronger version of any imported theorem is used.

\subsection{Levi graphs and expansions}

A hypergraph $H=(V(H),E(H))$ is \emph{linear} if any two distinct edges
meet in at most one vertex.  Its \emph{Levi graph} $I(H)$ is the
bipartite graph whose two vertex classes are $V(H)$ and $E(H)$, with a
point-node $v$ adjacent to a hyperedge-node $e$ exactly when $v\in e$.

A \emph{Berge cycle of length $m$} is a sequence
\[
 v_0,e_0,v_1,e_1,\ldots,v_{m-1},e_{m-1},v_0
\]
in which the $v_i$ are distinct point-nodes, the $e_i$ are distinct
hyperedge-nodes, and $v_i,v_{i+1}\in e_i$ for every $i$ modulo $m$.
Equivalently, it is a cycle of length $2m$ in $I(H)$.  In a linear
simple hypergraph every Berge cycle has length at least $3$.

For a finite graph $J$, its private-vertex expansion $J^+$ has the core
vertices of $J$ and, for each edge $xy\in E(J)$, one new vertex $p_{xy}$
used in no other expanded edge; its edges are $\{x,y,p_{xy}\}$.  A
one-point amalgamation of finite hypergraphs $F_0,F_1$ is obtained by
choosing $v_i\in V(F_i)$, taking otherwise disjoint copies, and
identifying $v_0$ with $v_1$.

For a finite hypergraph $F$, let $F^\circ$ denote the subhypergraph
obtained by deleting all isolated vertices.

\begin{lemma}[Isolated-vertex reduction]
\label{lem:isolated-reduction}
Let $F$ be finite and let $H$ be an infinite hypergraph.  Then
\[
  F\hookrightarrow H\quad\Longleftrightarrow\quad F^\circ\hookrightarrow H.
\]
Consequently $F$ is obligatory if and only if $F^\circ$ is obligatory,
$\Spec(F)=\Spec(F^\circ)$, and $F\in\Bclass$ if and only if
$F^\circ\in\Bclass$.
\end{lemma}

\begin{proof}
Only the reverse implication in the displayed equivalence requires
proof.  Fix a copy of $F^\circ$ in $H$.  Its image is finite, whereas
$H$ has infinitely many vertices, so there are enough unused vertices
to receive the finitely many isolated vertices of $F$.  Because
containment is non-induced, additional host edges among those vertices
are irrelevant.  Every host occurring in the definitions of
obligatoriness and $\Spec$ is infinite, so the first two consequences
follow.

We prove the assertion about $\Bclass$ in both directions.  If
$F^\circ\in\Bclass$, then $F$ is the disjoint union of $F^\circ$ and a
finite edgeless system, and hence $F\in\Bclass$.

Conversely, suppose $F\in\Bclass$.  We use structural induction on a
chosen construction of $F$ from the generators of $\Bclass$.  If
$F=J^+$, and $J^\circ$ is obtained from $J$ by deleting its isolated
graph vertices, then
\[
  (J^+)^\circ=(J^\circ)^+\in\Bclass.
\]
If $F$ is edgeless, then $F^\circ$ is the empty edgeless system and is
again a generator.  Reduction plainly commutes with disjoint union.

It remains to consider a one-point amalgamation of $F_0$ and $F_1$ in
which $x_i\in V(F_i)$ are identified.  By induction
$F_0^\circ,F_1^\circ\in\Bclass$.  If both $x_i$ are nonisolated in their
respective factors, then $F^\circ$ is the one-point amalgamation of
$F_0^\circ$ and $F_1^\circ$ at those points.  If exactly one $x_i$ is
nonisolated, then the identified point belongs only to the reduced
nonisolated factor, and $F^\circ$ is the disjoint union
$F_0^\circ\mathbin{\dot\cup}F_1^\circ$.  The same disjoint-union
description holds when both $x_i$ are isolated, because the identified
point is then deleted.  Every case is closed in $\Bclass$, completing
the induction.
\end{proof}

\subsection{The classical nonlinear obstruction}

\begin{theorem}[Erd\H{o}s--Hajnal--Rothschild]
\label{thm:EHR-nonlinear}
Let $F$ be a finite uniform hypergraph.  If two distinct edges of $F$
intersect in at least two vertices, then $F$ is non-obligatory.  In
particular, every obligatory finite triple system is linear.
\end{theorem}

\begin{proof}[Source]
This is the $i=2$ case of Erd\H{o}s--Hajnal--Rothschild
\cite[Theorem~2, p.~532]{EHR1973}; see also Reiher's explicit summary
\cite[p.~1]{ReiherObligatory}.
\end{proof}

\subsection{Exact high-odd-girth graphs}

\begin{theorem}[Erd\H{o}s--Hajnal]
\label{thm:EH-odd-girth}
For every uncountable cardinal $\kappa$ and every positive integer $s$,
there is a graph $A$ such that
\[
 |V(A)|=\chi(A)=\kappa
\]
and $A$ contains no odd cycle of length at most $2s+1$.  Here, and
throughout the paper, ``contains no $C_m$'' means that $A$ has no ordinary,
not necessarily induced, subgraph isomorphic to $C_m$.
\end{theorem}

\begin{proof}[Source]
Theorem~C on p.~428 of Erd\H{o}s--Galvin--Hajnal \cite{EGH},
quoting Erd\H{o}s--Hajnal \cite[Theorem~7.4, p.~76]{EH1966}, states that
for each positive integer $i$ one may exclude $C_{2j+1}$ for
$0<j<i$.  Taking $i=s+1$ gives exactly the displayed formulation.
\end{proof}

\section{The complete-rank one-apex lift}
\label{sec:lift}

Let $A$ be a graph and let $\kappa$ be an infinite cardinal.  Put
\[
 T(A,\kappa)=\bigcup_{\alpha<\kappa}E(A)^\alpha.
\]
Thus every $\sigma\in T(A,\kappa)$ is a sequence of edges of $A$ whose
ordinal length $\lh(\sigma)$ is less than $\kappa$.  We write
$\sigma\subsetneq\tau$ when $\sigma$ is a proper initial segment of
$\tau$.

\begin{definition}
The \emph{complete-rank one-apex lift} $\Lift(A,\kappa)$ is the triple
system with vertex set $T(A,\kappa)\times V(A)$.  Its edges are the
triples
\begin{equation}
 \{(\sigma,x),(\sigma,y),(\tau,z)\}\label{eq:lift-edge}
\end{equation}
for which
\[
 \sigma\subsetneq\tau,
 \qquad \tau(\lh(\sigma))=\{x,y\}\in E(A),
 \qquad z\in V(A).
\]
The first two vertices are the \emph{base}, the third is the
\emph{apex}, $\sigma$ is the \emph{source node}, and $\tau$ is the
\emph{apex node}.
\end{definition}

Every lift edge has sequence-node multiset $\{\sigma,\sigma,\tau\}$ with
$\sigma\subsetneq\tau$.  Hence its apex is determined uniquely by the
triple itself: it is the only vertex whose sequence node differs from the
other two.

\begin{theorem}
\label{thm:lift-chromatic}
If $\chi(A)=\kappa$, then
\[
 \chi(\Lift(A,\kappa))=\kappa.
\]
\end{theorem}

\begin{proof}
Let $d:V(A)\to\kappa$ be a proper colouring of $A$.  Colour
$(\sigma,x)$ by $d(x)$.  Every lift edge contains a complete edge of $A$
in its base, so this is proper.

Conversely, suppose
\[
 c:T(A,\kappa)\times V(A)\longrightarrow\theta,
 \qquad \theta<\kappa.
\]
Recursively construct a branch $t\in E(A)^\kappa$.  At stage $\alpha$,
the map
\[
 x\longmapsto c(t\restr\alpha,x)
\]
is a $\theta$-colouring of $A$ and hence is not proper.  Choose an edge
$t(\alpha)=\{x_\alpha,y_\alpha\}$ monochromatic in some colour
$d_\alpha$.  Since $\theta<\kappa$, choose $\alpha<\beta<\kappa$ with
$d_\alpha=d_\beta$, and choose an endpoint $z$ of $t(\beta)$.  Because
sequence length means ordinal domain,
\[
 (t\restr\beta)(\lh(t\restr\alpha))
   =(t\restr\beta)(\alpha)=t(\alpha).
\]
Therefore
\[
 \{(t\restr\alpha,x_\alpha),(t\restr\alpha,y_\alpha),
   (t\restr\beta,z)\}
\]
is a lift edge.  Its base has colour $d_\alpha$ and its apex has colour
$d_\beta=d_\alpha$, a contradiction.  Thus no colouring with fewer than
$\kappa$ colours is proper.  Notice that the argument used no regularity
or cofinality assumption on $\kappa$.
\end{proof}

\section{Finite linear traces of the lift}
\label{sec:trace}

Throughout this section, finite source systems have no isolated vertices.
This entails no loss by \Cref{lem:isolated-reduction}.

\begin{definition}
Let $F$ be a finite linear triple system.  A \emph{bridge selector} is a
map
\[
 p:E(F)\longrightarrow V(F),\qquad p(e)\in e,
\]
such that the Levi incidence $ep(e)$ is a bridge of $I(F)$ for every
$e\in E(F)$.

Delete all selected incidences.  If $C$ is a resulting component
containing at least one hyperedge-node, define the graph derivative
$D_C(F,p)$ to have vertex set $C\cap V(F)$ and, for every hyperedge-node
$e\in C$, the graph edge
\[
 e\setminus\{p(e)\}.
\]
These graph edges are distinct: if two source hyperedges had the same
surviving pair, they would share two vertices, contrary to linearity.
\end{definition}

\begin{lemma}[Cycle--selector correspondence]
\label{lem:cycle-selector}
Let $F$ be a finite linear triple system, let $p$ be a bridge selector,
and let
\[
 v_0,e_0,v_1,e_1,\ldots,v_{m-1},e_{m-1},v_0
\]
be a Berge cycle.  Then:
\begin{enumerate}[label=\textup{(\roman*)}]
\item neither displayed incidence $e_iv_i$ nor $e_iv_{i+1}$ is selected;
\item $p(e_i)\notin\{v_0,\ldots,v_{m-1}\}$ for every $i$;
\item after the selected incidences are deleted, the displayed Levi
cycle lies in one active component $C$;
\item $D_C(F,p)$ contains the ordinary graph cycle
\[
 v_0v_1\cdots v_{m-1}v_0.
\]
Conversely, every ordinary graph cycle in a derivative $D_C(F,p)$ gives
a Berge cycle of the same length in $F$.  Consequently, all derivatives
of $(F,p)$ are bipartite if and only if every Berge cycle of $F$ has
even length.
\end{enumerate}
\end{lemma}

\begin{proof}
Every displayed incidence belongs to the displayed Levi cycle, whereas
a selected incidence is a bridge.  This proves~\textup{(i)}.

Suppose $p(e_i)=v_j$ for some connector $v_j$.  The selected incidence
$e_iv_j$ is not one of the two displayed incidences at $e_i$ by
\textup{(i)}.  Starting at $e_i$, follow the displayed Levi cycle to
$v_j$ along either direction that does not use $e_iv_j$; adjoining the
incidence $e_iv_j$ produces a Levi cycle.  The selected incidence would
therefore not be a bridge.  This proves~\textup{(ii)}.

By~\textup{(i)}, every edge of the displayed Levi cycle survives the
deletion of the selected incidences.  Hence the cycle lies in one
component $C$, which is active because it contains the nodes $e_i$.
By~\textup{(ii)}, the selected point of $e_i$ is its third,
nonconnector point, and therefore the derivative edge contributed by
$e_i$ is precisely $\{v_i,v_{i+1}\}$.  The connector vertices and the
hyperedges are distinct by the definition of a Berge cycle.  Moreover,
linearity forbids two source hyperedges from producing the same surviving
pair.  The displayed derivative edges therefore form an actual copy of
$C_m$, proving~\textup{(iii)} and~\textup{(iv)}.

Conversely, let
\[
 x_0x_1\cdots x_{m-1}x_0
\]
be a graph cycle in $D_C(F,p)$.  For every $i$, let $e_i$ be the unique
source hyperedge whose derivative edge is $\{x_i,x_{i+1}\}$.  Distinct
derivative edges come from distinct hyperedges, and the $x_i$ are
distinct.  Hence
\[
 x_0,e_0,x_1,e_1,\ldots,x_{m-1},e_{m-1},x_0
\]
is a Berge cycle of length $m$ in $F$.  The final equivalence follows
from the ordinary characterization of bipartite graphs by absence of odd
cycles.
\end{proof}

\begin{lemma}[Cycle collapse]
\label{lem:cycle-collapse}
Let $F$ be finite and linear, and let
$\varphi:F\hookrightarrow\Lift(A,\kappa)$ be an embedding.  Write
\[
 \varphi(v)=(\nu(v),a(v)).
\]
On every Berge cycle
\[
 v_0,e_0,v_1,e_1,\ldots,v_{m-1},e_{m-1},v_0
\]
of $F$, all sequence nodes $\nu(v_i)$ are equal.
\end{lemma}

\begin{proof}
Since $F$ is linear, $m\ge3$.  Consecutive nodes
$\nu(v_i),\nu(v_{i+1})$ are equal or comparable in the initial-segment
order, because the two vertices lie in one lift edge.  Let
\[
 d=\min_i\lh(\nu(v_i))
\]
and choose $j$ attaining the minimum.  Put $\sigma=\nu(v_j)$.

We first record an elementary prefix observation.  Suppose
$\sigma_0,\ldots,\sigma_t$ are sequences such that consecutive terms are
comparable, every term has ordinal length at least $d$, and
$\lh(\sigma_0)=d$.  If $\sigma_i$ extends $\sigma_0$ and
$\sigma_{i+1}$ is comparable with $\sigma_i$, then either
$\sigma_i\subseteq\sigma_{i+1}$ or $\sigma_{i+1}$ is a prefix of
$\sigma_i$ of length at least $d$.  In both cases $\sigma_{i+1}$ extends
$\sigma_0$.  Traversing the cycle from $v_j$ therefore shows that every
$\nu(v_i)$ extends $\sigma$.  In particular, every connector node of
length $d$ is exactly $\sigma$.

Suppose some connector node properly extends $\sigma$.  In cyclic order,
choose a maximal nonempty interval of connector vertices whose nodes
properly extend $\sigma$, bounded at both ends by connector nodes equal
to $\sigma$.  Write its node sequence as
\[
 \sigma,\tau_1,\ldots,\tau_s,\sigma.
\]
If the chosen occurrence of $\sigma$ is the only connector on the cycle
whose node equals $\sigma$, the two boundary hyperedges are nevertheless
distinct: they are the two distinct cycle hyperedges incident with that
connector, and a Berge cycle has length at least $3$.  Every internal
$\tau_i$ is a proper extension of $\sigma$, and consecutive internal
nodes are comparable.  Hence all $\tau_i$ have the same value at
coordinate $d$; denote this graph edge by $a\in E(A)$.

At the first boundary hyperedge, the connector at $\tau_1$ is the unique
apex, because the other boundary connector has node $\sigma\ne\tau_1$.
Thus the connector at $\sigma$ and the third vertex of the source
hyperedge form the base at $\sigma$, whose image is the unordered host
pair
\[
 \{(\sigma,x),(\sigma,y)\},\qquad \{x,y\}=a.
\]
The return boundary hyperedge has the same host base pair, since
$\tau_s(d)=a$.  Both boundary image edges therefore contain exactly the
same two base vertices $(\sigma,x)$ and $(\sigma,y)$.  Since $\varphi$ is
injective on vertices, each of these two host vertices has a unique source
preimage.  Hence the two distinct source hyperedges contain the same two
source vertices, contradicting linearity.  Therefore no connector node
properly extends $\sigma$, and all connector nodes are equal.
\end{proof}

\begin{lemma}[Quotient forest]
\label{lem:quotient-forest}
Let $G$ be a graph and let $S\subseteq E(G)$ consist of bridges of $G$.
Contract every component of $G-S$ and retain the edges in $S$.  The
resulting quotient is a simple forest.
\end{lemma}

\begin{proof}
Every edge of $S$ joins two distinct components of $G-S$.  If the
quotient contained a cycle, including a pair of parallel quotient edges,
then every selected edge on that cycle would have an alternative path
between its endpoints in $G$, obtained by expanding the other quotient
edges and the contracted components.  It would not be a bridge, a
contradiction.
\end{proof}

\begin{lemma}[Labelled forest lemma]
\label{lem:labelled-forest}
Let $T$ be a finite forest.  Orient its edges arbitrarily and label every
oriented edge by an element of a nonempty alphabet $\Omega$.  There are
pairwise distinct finite words $(w_x:x\in V(T))$ over $\Omega$ such that,
for every oriented edge $x\to y$ with label $a$,
\[
 w_x\subsetneq w_y
 \qquad\text{and}\qquad
 w_y(\lh(w_x))=a.
\]
\end{lemma}

\begin{proof}
First treat a tree, by induction on its number of vertices.  The
one-vertex case is immediate.  Let $z$ be a leaf with neighbour $y$, and
apply the induction hypothesis to $T-z$, obtaining old words
$(u_x:x\in V(T-z))$.

If the new edge is oriented $y\to z$ and has label $a$, retain
$w_x=u_x$ for $x\ne z$ and put
\[
 w_z=u_y\concat\langle a\rangle\concat u,
\]
where the finite padding word $u$ is chosen so that $w_z$ has ordinal
length larger than every old word.  The required symbol occurs at
coordinate $\lh(w_y)=\lh(u_y)$, and all words remain distinct.

If the new edge is oriented $z\to y$ and has label $a$, define
\[
 w_z=\varnothing,
 \qquad
 w_x=\langle a\rangle\concat u_x\quad(x\ne z).
\]
For an old oriented edge $x\to x'$ we have
\[
 w_{x'}(\lh(w_x))
 =w_{x'}(1+\lh(u_x))
 =u_{x'}(\lh(u_x)),
\]
so its prescribed symbol is unchanged.  The new edge condition holds at
coordinate $0$, and distinctness is preserved.

For a forest, perform this construction in each tree component.  Fix
$a_0\in\Omega$.  Process the components one at a time and prefix every
word in the next component by $a_0^N$, choosing $N$ so that the finite
set of new word lengths is disjoint from every length used earlier.
Common prefixing preserves all internal edge conditions, while disjoint
sets of lengths guarantee global distinctness.
\end{proof}

\begin{theorem}[Bridge-trace theorem]
\label{thm:bridge-trace}
Let $F$ be a finite linear triple system without isolated vertices and
with $E(F)\ne\varnothing$, let $A$ be a graph with at least one edge, and
let $\kappa$ be infinite.
Then
\[
 F\hookrightarrow\Lift(A,\kappa)
\]
if and only if $F$ has a bridge selector $p$ such that every derivative
$D_C(F,p)$ embeds in $A$.
\end{theorem}

\begin{proof}[Necessity]
Suppose $\varphi:F\hookrightarrow\Lift(A,\kappa)$.  Every image edge has
a unique apex.  For each $e\in E(F)$, let $p(e)$ be the vertex whose
image is that apex.

We claim that the incidence $ep(e)$ is a bridge of $I(F)$.  In a finite
graph an edge is a bridge if and only if it lies on no cycle: if deleting
an edge leaves a path between its endpoints, that path together with the
edge is a cycle, and the converse is immediate.  Thus, if $ep(e)$ were
not a bridge, it would lie on a Levi cycle.  That cycle is a Berge cycle
and, at the
hyperedge-node $e$, uses the incidence to $p(e)$ and another incidence
to a base vertex.  By \Cref{lem:cycle-collapse}, those two point-vertices
have the same sequence node.  This contradicts the strict prefix
relation between an apex node and a base node.  Thus $p$ is a bridge
selector.

Remove the selected incidences.  Each surviving incidence joins an
edge-node to one of its two base vertices.  Consequently, along any
path in a component $C$ containing hyperedge-nodes, all point-nodes have
one common sequence node, say $\sigma_C$.  The second-coordinate
projection is injective on those point-nodes because the original
embedding is injective and their first coordinates agree.  For each $e\in C$, it sends the two vertices other than $p(e)$ to adjacent vertices of $A$.
Two derivative edges cannot collapse to one graph edge: their source
pairs are distinct by linearity, and the point map is injective.
Therefore the projection embeds $D_C(F,p)$ in $A$.
\end{proof}

\begin{proof}[Sufficiency]
Suppose $p$ is a bridge selector and, for every component $C$ containing
hyperedge-nodes, fix an embedding
\[
 \psi_C:D_C(F,p)\hookrightarrow A.
\]
Delete the selected incidences and contract the resulting components.
By \Cref{lem:quotient-forest}, the quotient $T$ is a finite simple forest.
Orient the quotient edge arising from $ep(e)$ from the component $C_e$
containing the hyperedge-node $e$ towards the component $C_{p(e)}$
containing the point-node $p(e)$, and label it by
\[
 \psi_{C_e}(e\setminus\{p(e)\})\in E(A).
\]
Apply \Cref{lem:labelled-forest} with alphabet $E(A)$, obtaining a
distinct finite word $w_C$ for every quotient component $C$.

If a point-node $v$ belongs to a component $C$ containing
hyperedge-nodes, define
\[
 \Phi(v)=(w_C,\psi_C(v)).
\]
A component with no hyperedge-node is a singleton point-node; for such a
point $v$, choose any $a_v\in V(A)$ and put
$\Phi(v)=(w_C,a_v)$.  The map $\Phi$ is injective: inside an active
component this follows from injectivity of $\psi_C$, and between
components it follows from distinctness of the words.

Let $e=\{x,y,p(e)\}$.  Its two base points $x,y$ lie in $C_e$, and
\[
 \{\psi_{C_e}(x),\psi_{C_e}(y)\}
   =\psi_{C_e}(e\setminus\{p(e)\}).
\]
The labelled forest condition gives
$w_{C_e}\subsetneq w_{C_{p(e)}}$ and puts this graph edge at coordinate
$\lh(w_{C_e})$.  Hence
\[
 \{\Phi(x),\Phi(y),\Phi(p(e))\}
\]
is a lift edge.  Every $w_C$ is finite, so $\lh(w_C)<\kappa$ for every
infinite $\kappa$.  Thus $\Phi$ embeds $F$.
\end{proof}

\begin{remark}[Isolated vertices]
\label{rem:bridge-trace-isolated}
Every lift in \Cref{thm:bridge-trace} is infinite.  By
\Cref{lem:isolated-reduction}, adjoining or deleting finitely many isolated
source vertices does not affect embeddability in such a lift.  Thus the
bridge-trace theorem extends verbatim to an arbitrary finite linear $F$
after $F$ is replaced by $F^\circ$; the empty reduced system is handled
vacuously.
\end{remark}

\section{Obligatory triple systems and Erd\H{o}s Problem~\#593}
\label{sec:classification}

We now prove the finite decomposition.  The running-intersection lemma is
stated separately because the fact that the quotient is a forest does
not, by itself, identify how the corresponding expansion pieces
intersect.

By \Cref{lem:isolated-reduction}, isolated vertices may be removed and
restored at the end.  Fix a finite linear triple system $F$ without
isolated vertices and a bridge selector $p$.  Let
\[
 S_p=\{ep(e):e\in E(F)\},
\]
and let $\mathcal C$ be the set of components of $I(F)-S_p$.  A component
is \emph{active} if it contains a hyperedge-node.  For active $C$, let
$E_C$ be the set of hyperedges whose hyperedge-nodes lie in $C$, and let
$F_C$ be the subhypergraph with edge set $E_C$ and vertex set
$\bigcup E_C$.

\begin{lemma}[Expansion pieces]
\label{lem:expansion-piece}
For every active component $C$,
\[
 F_C\cong D_C(F,p)^+.
\]
\end{lemma}

\begin{proof}
Fix $e\in E_C$.  The selected point $p(e)$ lies outside $C$, because the
selected bridge $ep(e)$ has its endpoints in distinct components after
it is deleted.  We prove that $p(e)$ is private in the piece $F_C$ by
considering the two possible ways in which it could occur on another
edge $f\in E_C$.

First suppose that $p(e)$ were a base point of $f$.  Then the incidence
$fp(e)$ survives the deletion, so the point-node $p(e)$ lies in the same
component $C$ as the hyperedge-node $f$.  This contradicts the fact that
the endpoints of the deleted bridge $ep(e)$ lie in distinct components.

Second suppose that $p(e)=p(f)$ for some distinct $f\in E_C$.  Since
$e$ and $f$ lie in the connected component $C$, there is an $e$--$f$
path inside $I(F)-S_p$.  Together with the two selected incidences
through their common point, this path gives an alternative route around
each selected incidence.  Neither selected incidence would be a bridge,
a contradiction.

Thus the selected points $p(e)$, $e\in E_C$, are pairwise distinct and
each occurs on only its own edge of $F_C$.  The remaining pair
$e\setminus\{p(e)\}$ is exactly the corresponding edge of
$D_C(F,p)$.  This identifies $F_C$ with the private-vertex expansion
$D_C(F,p)^+$.
\end{proof}

\begin{lemma}[Leaf-piece or running-intersection lemma]
\label{lem:leaf-piece}
Let $T$ be the quotient forest obtained by contracting the components of
$I(F)-S_p$.  Root each tree component of $T$.  Order the active
components by nondecreasing distance from the root, breaking ties
arbitrarily.  When an active piece $F_C$ is added, its vertex set meets
the union of all previously added active pieces in at most one point.
\end{lemma}

\begin{proof}
For a point $v\in V(F)$, let $Q(v)\in\mathcal C$ be the component
containing its point-node.  The definition of $F_C$ gives the following exact membership
criterion.  A point $v$ belongs to $V(F_C)$ if and only if one of these
two alternatives holds:
\begin{enumerate}[label=\textup{(\alph*)},leftmargin=2.2em]
\item $Q(v)=C$;
\item $Q(v)$ is adjacent to $C$ in $T$, and the quotient edge $CQ(v)$
comes from a selected incidence $ep(e)$ with $e\in E_C$ and $p(e)=v$.
\end{enumerate}
Indeed, a base point of an edge in $E_C$ remains joined to its
hyperedge-node and therefore lies in $C$, whereas the selected point is
the point endpoint of the quotient edge corresponding to that selected
incidence.

Let $D$ be an earlier active component and suppose
$v\in V(F_C)\cap V(F_D)$.  Put $r=d_T(C)$, where $d_T$ denotes distance
from the chosen root.  Since $D$ was added no later than $C$, we have
$d_T(D)\le r$.

If $Q(v)=C$, then the membership criterion, applied to $F_D$, shows that
$D$ is adjacent to $C$.  Among the neighbours of $C$, the only one of
depth at most $r$ is its parent.  Hence $D$ is the parent of $C$, and
$v$ is the point endpoint of the unique parent edge of $C$.

Suppose instead that $Q(v)\ne C$.  Then the membership criterion shows that
$Q(v)$ is a neighbour of $C$.  It cannot be a child of $C$.  Such a
child has depth $r+1$, and in a rooted tree its only neighbour of depth
at most $r$ is $C$ itself; no distinct earlier active component $D$
could then also contain $v$.  Therefore $Q(v)$ is the parent of $C$.
Again the unique parent edge of $C$ has point endpoint $v$.

Thus every point shared with an earlier active piece is the point
endpoint of the unique quotient edge joining $C$ to its parent.  This
quotient edge represents one selected incidence and hence has exactly
one point endpoint.  All previously shared points are therefore equal.
If $C$ is a root, no such point exists and the intersection is empty.
\end{proof}

\begin{example}[A nontrivial quotient tree]
\label{ex:worked-decomposition}
Let $F$ have edges
\[
\begin{aligned}
 e_0&=\{a,b,v\},& e_1&=\{v,c,u\},& e_2&=\{c,h,d\},\\
 e_3&=\{d,e,q\},& e_4&=\{f,g,q\},
\end{aligned}
\]
and choose
\[
 p(e_0)=v,\quad p(e_1)=u,\quad p(e_2)=d,\quad
 p(e_3)=q,\quad p(e_4)=q.
\]
Every selected incidence is a Levi bridge.  After deletion, the active
components are $C_0,C_1,C_2,C_3$.  The inactive point components are
$U=\{u\}$ and $Q=\{q\}$.  The derivative edge sets are
\[
 E(D_{C_0})=\{ab\},\quad E(D_{C_1})=\{vc,ch\},\quad
 E(D_{C_2})=\{de\},\quad E(D_{C_3})=\{fg\}.
\]
Thus $F_{C_1}$ is the expansion of the two-edge path $v-c-h$, while the
other active pieces are one-edge expansions.
\end{example}

\begin{figure}[htbp]
\centering
\resizebox{\textwidth}{!}{%
\begin{tikzpicture}[
  point/.style={circle,draw,minimum size=5.5mm,inner sep=0pt},
  hedge/.style={draw,rounded corners=1pt,minimum width=7mm,minimum height=5.5mm},
  active/.style={draw,rounded corners,minimum width=10mm,minimum height=6mm},
  inactive/.style={circle,draw,minimum size=6mm,inner sep=0pt},
  selected/.style={densely dashed,line width=.75pt},
  qedge/.style={-{Latex[length=2mm]},line width=.55pt},
  every node/.style={font=\small}
]
  \node[hedge] (E0) at (0,0) {$e_0$};
  \node[point] (v) at (1.5,0) {$v$};
  \node[hedge] (E1) at (3,0) {$e_1$};
  \node[point] (c) at (4.5,0) {$c$};
  \node[hedge] (E2) at (6,0) {$e_2$};
  \node[point] (d) at (7.5,0) {$d$};
  \node[hedge] (E3) at (9,0) {$e_3$};
  \node[point] (q) at (10.5,0) {$q$};
  \node[hedge] (E4) at (12,0) {$e_4$};
  \node[point] (a) at (-.65,.95) {$a$};
  \node[point] (b) at (-.65,-.95) {$b$};
  \node[point] (u) at (3,1.05) {$u$};
  \node[point] (h) at (6,1.05) {$h$};
  \node[point] (ee) at (9,1.05) {$e$};
  \node[point] (f) at (12.65,.95) {$f$};
  \node[point] (g) at (12.65,-.95) {$g$};
  \draw (E0)--(a) (E0)--(b);
  \draw[selected] (E0)--(v);
  \draw (v)--(E1) (E1)--(c);
  \draw[selected] (E1)--(u);
  \draw (c)--(E2) (E2)--(h);
  \draw[selected] (E2)--(d);
  \draw (d)--(E3) (E3)--(ee);
  \draw[selected] (E3)--(q);
  \draw[selected] (q)--(E4);
  \draw (E4)--(f) (E4)--(g);
  \node[anchor=west] at (-.9,1.65) {Levi graph; dashed incidences are selected};

  \node[active] (C0) at (2.1,-3.1) {$C_0$};
  \node[active] (C1) at (4.4,-3.1) {$C_1$};
  \node[inactive] (U) at (4.4,-4.45) {$U$};
  \node[active] (C2) at (6.7,-3.1) {$C_2$};
  \node[inactive] (Q) at (9.0,-3.1) {$Q$};
  \node[active] (C3) at (11.3,-3.1) {$C_3$};
  \draw[qedge] (C0)--node[above] {$v$}(C1);
  \draw[qedge] (C1)--node[left] {$u$}(U);
  \draw[qedge] (C1)--node[above] {$d$}(C2);
  \draw[qedge] (C2)--node[above] {$q$}(Q);
  \draw[qedge] (C3)--node[above] {$q$}(Q);
  \node[anchor=west] at (1.2,-2.35) {Quotient forest, oriented toward the selected point component};
\end{tikzpicture}}
\caption{The Levi graph and quotient forest from
\Cref{ex:worked-decomposition}.}
\label{fig:worked-decomposition}
\end{figure}
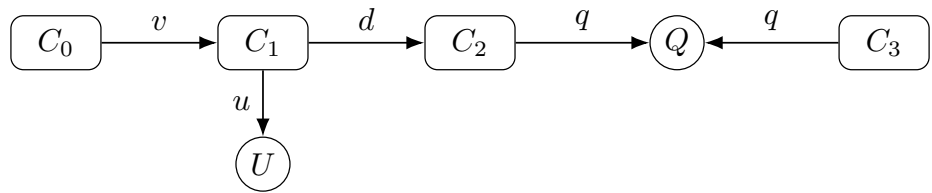

Rooting the quotient at $C_1$, one reconstructs $F$ by taking
$F_{C_1}$ first, amalgamating $F_{C_0}$ at $v$, amalgamating $F_{C_2}$
at $d$, and finally amalgamating $F_{C_3}$ at $q$.  The inactive
component $U$ contributes no edge.  This example exhibits attachment
through both an active component and an inactive singleton point
component.

\begin{proposition}[Finite bridge decomposition]
\label{prop:finite-decomposition}
For a finite triple system $F$, the following are equivalent after
isolated vertices are removed.
\begin{enumerate}[label=\textup{(\roman*)}]
\item $F\in\Bclass$;
\item $F$ is linear, every hyperedge-node of $I(F)$ is incident with a
bridge, and every Berge cycle of $F$ has even length.
\end{enumerate}
\end{proposition}

\begin{proof}
Assume first that $F=J^+$ for a finite bipartite graph $J$.  The
incidence from every expanded edge to its private vertex is a bridge.
Every Berge cycle uses the two core incidences at each hyperedge-node and
suppresses to a graph cycle in $J$, so it has even length.  Linearity is
immediate.  Every finite edgeless system also satisfies the three
conditions after its isolated vertices are removed, since its reduction
is the empty system.  Disjoint unions preserve the three properties.  A one-point
amalgamation joins two Levi graphs at one point-node.  No cycle can use
edges from both sides, and a bridge from either factor cannot acquire an
alternative path through the other factor, which meets it only at the
amalgamating point.  Thus every member of $\Bclass$ satisfies~\textup{(ii)}.

Conversely, suppose~\textup{(ii)}.  At every hyperedge-node choose one
incident bridge, obtaining a bridge selector $p$.  Delete the selected
incidences.  For every active component $C$, form
$D_C=D_C(F,p)$.  Every $D_C$ is bipartite by \Cref{lem:cycle-selector}: an odd graph
cycle in a derivative would give an odd Berge cycle of the same length in
$F$.

By \Cref{lem:expansion-piece}, $F_C\cong D_C^+$, so every active piece is
a generator of $\Bclass$.  The active pieces contain all edges and,
because $F$ has no isolated vertices, their vertex sets cover $V(F)$.
Root the quotient forest and list the active pieces in the order from
\Cref{lem:leaf-piece}.  We maintain the following reconstruction
invariant: after the first $j$ active pieces have been added, the
constructed hypergraph is exactly the union of those $j$ source
subhypergraphs, with precisely their original vertex identifications in
$F$.  The invariant is immediate for the first piece.  Suppose it holds
through stage $j$.  The next piece $F_C$ meets the existing union in a set
of size $0$ or $1$ by \Cref{lem:leaf-piece}.  In the first case we take a
disjoint union.  In the second we identify the unique shared source
vertex and take a one-point amalgamation.  The resulting hypergraph is
therefore exactly the union of the first $j+1$ pieces with their original
identifications, so the invariant continues.  Even if several earlier
pieces already contain the same inactive attachment point, the new step
identifies only that one existing vertex.  Different quotient-tree
components are finally combined by disjoint union.  Hence $F\in\Bclass$.
\end{proof}

For completeness, we prove the closure property needed for the positive
direction.  The proof is elementary apart from the de Bruijn--Erd\H{o}s
compactness theorem for finite graph colourings.

\begin{lemma}
\label{lem:obligatory-closure}
The class of finite obligatory triple systems is closed under taking
subhypergraphs, finite disjoint unions, adjoining isolated vertices, and
one-point amalgamations.
\end{lemma}

\begin{proof}
Subhypergraphs and isolated vertices are immediate.  If $F_0,F_1$ are
obligatory and $H$ has uncountable chromatic number, find a copy of
$F_0$.  Deleting its finite vertex set leaves an uncountably chromatic
hypergraph; otherwise a countable colouring of the remainder, together
with finitely many fresh colours, would colour $H$.  Find a disjoint copy
of $F_1$ in the remainder.  This proves closure under finite disjoint
unions.

Now choose roots $r_i\in V(F_i)$ and let $F$ be their one-point
amalgamation.  Suppose, towards a contradiction, that an uncountably
chromatic hypergraph $H$ is $F$-free.  Let $B$ be the set of vertices
$v\in V(H)$ that can serve as the image of $r_0$ in a copy of $F_0$, and
put $A=V(H)\setminus B$.  Then $H[A]$ is $F_0$-free and hence countably
colourable.

For each $v\in B$, choose a rooted copy $K_v$ of $F_0$ with root $v$.
Define an auxiliary simple graph $D$ on $B$ by joining distinct $v,w$
when $w\in V(K_v)\setminus\{v\}$ or
$v\in V(K_w)\setminus\{w\}$.  Put $d=|V(F_0)|-1$.  Assign every
auxiliary edge to one endpoint whose selected rooted copy generated it,
and orient it away from that endpoint.  Each vertex has outdegree at
most $d$.  Consequently, for every finite $U\subseteq B$,
\[
  |E(D[U])|\le d|U|.
\]
The same estimate applies to $D[W]$ for every nonempty $W\subseteq U$.
Thus every such $D[W]$ has average degree at most $2d$ and hence has a
vertex of degree at most $2d$.  It follows, by repeatedly deleting such
a vertex inside the finite graph $D[U]$, that $D[U]$ is
$2d$-degenerate and therefore $(2d+1)$-colourable.  Every finite subgraph
of $D$ is consequently $(2d+1)$-colourable.  The de Bruijn--Erd\H{o}s
compactness theorem for graph colouring \cite{deBruijnErdos} now gives a
proper $(2d+1)$-colouring of all of $D$.

Let $C$ be one colour class of $D$.  We claim that $H[C]$ is $F_1$-free.
If it contained a copy of $F_1$ whose root is mapped to $v\in C$, then
$V(K_v)\setminus\{v\}$ would be disjoint from $C$: points outside $B$
are not in $C$, and points in
$B\cap(V(K_v)\setminus\{v\})$ are adjacent to $v$ in $D$.  The two
rooted copies would therefore meet exactly at $v$, giving a copy of $F$,
a contradiction.  Since $F_1$ is obligatory, each $H[C]$ is countably
colourable.  Colour $H[A]$ and the finitely many classes $H[C]$ with
pairwise disjoint countable palettes.  A crossing edge receives colours
from at least two palettes, while an edge inside one part is
nonmonochromatic by construction.  This countably colours $H$, a
contradiction.
\end{proof}

\begin{theorem}[Reiher]
\label{thm:Reiher}
For every positive integer $n$, the triple-system expansion
$K_{n,n}^+$ is obligatory.
\end{theorem}

\begin{proof}[Source]
This is the case $k=3$ of Reiher's theorem on obligatory expansions
\cite[Theorem~1.2]{ReiherObligatory}.
\end{proof}

\begin{theorem}[Resolution of Erd\H{o}s Problem~\#593]
\label{thm:classification}
For every finite triple system $F$, the following are equivalent.
\begin{enumerate}[label=\textup{(\roman*)}]
\item $F$ is obligatory;
\item $F\in\Bclass$;
\item after isolated vertices are removed, $F$ is linear, every
hyperedge-node of $I(F)$ is incident with a bridge, and every Berge cycle
of $F$ has even length.
\end{enumerate}
\end{theorem}

\begin{proof}
Let $F^\circ$ be obtained from $F$ by deleting its isolated vertices.
By \Cref{lem:isolated-reduction}, all three assertions are unchanged when
$F$ is replaced by $F^\circ$.  We therefore assume that $F$ has no
isolated vertices.  If $E(F)=\varnothing$, then necessarily
$V(F)=\varnothing$; the empty system is a generator of $\Bclass$, embeds
in every host, and satisfies the intrinsic condition vacuously.  Hence
we may also assume $E(F)\ne\varnothing$.

The equivalence of \textup{(ii)} and \textup{(iii)} is
\Cref{prop:finite-decomposition}.

Assume \textup{(ii)}.  If $J$ is finite and bipartite, then
$J\hookrightarrow K_{n,n}$ for some $n$, so
$J^+\hookrightarrow K_{n,n}^+$.  By \Cref{thm:Reiher} and closure under
subhypergraphs, $J^+$ is obligatory.  Now
\Cref{lem:obligatory-closure} implies that every member of $\Bclass$ is
obligatory.  Thus \textup{(ii)} implies \textup{(i)}.

It remains to show that failure of \textup{(iii)} implies
non-obligatoriness.

If $F$ is nonlinear, two of its edges meet in at least two vertices, and
\Cref{thm:EHR-nonlinear} says directly that $F$ is non-obligatory.

Assume $F$ is linear but some hyperedge-node of $I(F)$ has no incident
bridge.  Then $F$ has no bridge selector.  By
\Cref{thm:lift-chromatic}, $\Lift(K_{\aleph_1},\aleph_1)$ has chromatic
number $\aleph_1$, while \Cref{thm:bridge-trace} shows that it omits $F$.

Finally, suppose $F$ is linear, every hyperedge-node has an incident
bridge, and $F$ has an odd Berge cycle of length $m\ge3$.  Choose, by
\Cref{thm:EH-odd-girth}, an exact-$\aleph_1$ graph $A$ with no $C_m$.
Then $\Lift(A,\aleph_1)$ has chromatic number $\aleph_1$.  If $F$
embedded in it, \Cref{thm:bridge-trace} would give a bridge selector all
of whose derivatives embed in $A$, whereas \Cref{lem:cycle-selector}
forces one derivative to contain an actual $C_m$.  This contradiction
shows that the lift omits $F$.

Every failure of \textup{(iii)} has produced an uncountably chromatic
$F$-free host, so \textup{(i)} implies \textup{(iii)}.
\end{proof}

\begin{corollary}[Obstruction trichotomy]
\label{cor:obstruction-trichotomy}
Let $F$ be a finite triple system and let $F^\circ$ be obtained by deleting
its isolated vertices.  Then $F\notin\Bclass$ if and only if exactly one of
the following sequentially exclusive alternatives holds:
\begin{enumerate}[label=\textup{(\roman*)}]
\item $F^\circ$ is nonlinear;
\item $F^\circ$ is linear and some hyperedge-node of $I(F^\circ)$ has no
incident bridge, equivalently $F^\circ$ has no bridge selector;
\item $F^\circ$ is linear, every hyperedge-node of $I(F^\circ)$ has an
incident bridge, and $F^\circ$ has an odd Berge cycle.
\end{enumerate}
\end{corollary}

\begin{proof}
By \Cref{prop:finite-decomposition}, membership in $\Bclass$ is equivalent
to the conjunction of linearity, the existence of an incident bridge at
every hyperedge-node, and evenness of every Berge cycle.  Negating this
conjunction in the displayed order gives the three alternatives.  They
are mutually exclusive because each later alternative includes the
negations of all earlier ones.
\end{proof}

A triple system is \emph{strongly tripartite} if its vertex set can be
partitioned into three classes and every edge meets each class exactly
once.  A finite triple-system \emph{forest} is a system whose nonisolated
edges can be ordered $e_0,\ldots,e_{t-1}$ so that
\[
  \left|e_i\cap\bigcup_{j<i}e_j\right|\le1
  \qquad(0<i<t).
\]

\begin{corollary}[Compatibility with the known theory]
\label{cor:compatibility}
The classification has the following consequences.
\begin{enumerate}
\item Every obligatory finite triple system is strongly tripartite.
\item Every finite triple-system forest is obligatory.
\item For every $n\ge3$, the private-vertex cycle expansion $C_n^+$ is
obligatory if and only if $n$ is even.
\item The loose cycle $C_7^+=C_7^{(3)}$ is linearly obligatory but is not
obligatory.
\end{enumerate}
\end{corollary}

\begin{proof}
For a bipartite graph $J$ with bipartition $X\mathbin{\dot\cup}Y$, the
three classes $X$, $Y$, and the set of private expansion vertices form a
strong tripartition of $J^+$.  Strong tripartiteness is preserved under
disjoint union.  It is also preserved under one-point amalgamation:
permute the three class names in one factor so that the two points to be
identified lie in corresponding classes, and then unite corresponding
classes.  Edgeless systems cause no restriction.  Thus every member of
$\Bclass$, and hence every obligatory system by
\Cref{thm:classification}, is strongly tripartite.  This recovers
Komj\'ath's necessary condition \cite[pp.~233--238]{Komjath2001}.

A single triple is $K_2^+$.  In the defining edge order of a forest, a
new edge is either disjoint from the previous edge union or meets it in
one point.  After isolated vertices are set aside, the forest is
therefore built from single triples by successive disjoint unions and
one-point amalgamations.  It lies in $\Bclass$ and is obligatory.

If $n$ is even, $C_n$ is bipartite, so $C_n^+$ is one of the generators
of $\Bclass$.  If $n$ is odd, its private incidences are bridges and its
core vertices and hyperedge-nodes form an odd Berge cycle.  It therefore
fails the intrinsic criterion in \Cref{thm:classification} and is not
obligatory.  Here a finite triple system is \emph{linearly obligatory} if it occurs
in every linear triple system of uncountable chromatic number.  Hajnal and
Komj\'ath prove that $C_n^{(3)}$ is linearly obligatory for
$n\notin\{2,3,5\}$ \cite{HK2008}; the precise statement, with this
terminology and the same loose-cycle convention, is recorded in
\cite[\S3.7, p.~43]{ReiherGirth}.  Taking $n=7$ and combining that result
with the preceding paragraph proves the last assertion.
\end{proof}

\begin{remark}[Recognition]
\label{rem:recognition}
Work in the incidence-list model and suppose linearity has already been
certified.  The criterion in \Cref{thm:classification} is recognisable in
\[
 O\bigl(|V(I(F))|+|E(I(F))|\bigr)
\]
time.  Compute all bridges of the Levi graph, verify that every
hyperedge-node has one, select one bridge at each hyperedge-node, delete
them, find the remaining components, build the graph derivatives, and
test every derivative for bipartiteness.  By
\Cref{lem:cycle-selector}, the derivative bipartiteness tests are exactly
equivalent to the even-Berge-cycle condition.  This claim deliberately
begins after linearity is supplied; testing repeated vertex pairs is a
separate, representation-dependent step.
\end{remark}

\part{Exact calibration and avoidance spectra}

\section{Exact linear calibration}
\label{sec:linear}

The classification in \Cref{thm:classification} uses only the classical
nonlinearity obstruction.  To obtain exact avoidance spectra, and hence
Problem~\#1177, we need the stronger fact that linear triple systems
exist at every prescribed uncountable chromatic cardinal.  This section
is logically independent of the bridge classification.

\subsection{The simultaneous edge-labelling input}

\begin{definition}
Let $S$ be a graph, let $I$ be a set, and let $\kappa$ be a cardinal.
An edge labelling $\ell:E(S)\to I$ has property $P(S,I,\kappa)$ if, for
every colouring $c:V(S)\to\theta$ with $\theta<\kappa$, there is a
colour $a<\theta$ such that, for every $i\in I$, the colour class
$c^{-1}\{a\}$ contains an edge $e$ with $\ell(e)=i$.
\end{definition}

\begin{lemma}[Monotonicity]
\label{lem:P-monotone}
If $P(S,I,\delta)$ holds and $\kappa\le\delta$, then $P(S,I,\kappa)$
holds.
\end{lemma}

\begin{proof}
Every colouring with fewer than $\kappa$ colours is also a colouring with
fewer than $\delta$ colours.
\end{proof}

\begin{theorem}[Erd\H{o}s--Galvin--Hajnal]
\label{thm:EGH-P}
Let $\rho$ be an infinite cardinal and let
\[
 \delta(\rho)=\min\{\delta:\rho^\delta>\rho\}.
\]
The generalized Specker graph $S=GS_2(\rho)$ has $|V(S)|=\rho$ and
there is a map
\[
  f:E(S)\longrightarrow\rho
\]
with the following simultaneous-labelling property: whenever
$\theta<\delta(\rho)$ and $c:V(S)\to\theta$, there is one colour
$a<\theta$ such that, for every $\xi<\rho$, some edge $xy\in E(S)$
satisfies
\[
  f(xy)=\xi
  \qquad\text{and}\qquad
  c(x)=c(y)=a.
\]
Equivalently, if $I$ is any set of cardinality $\rho$, then composing
$f$ with a bijection $\rho\to I$ gives an $I$-valued edge labelling with
property $P(S,I,\delta(\rho))$.
\end{theorem}

\begin{proof}[Source]
Definition~6.2 on p.~448 of Erd\H{o}s--Galvin--Hajnal \cite{EGH}
defines $P$ by exactly the displayed common-colour-class quantifier.
Their Corollary~9.7 on p.~461 supplies the assertion for the generalized
Specker graph; here we take its case $n=2$.  We use property $P$, not the
stronger property $P^*$.
\end{proof}

Since $S$ is a simple graph on the infinite cardinal $\rho$, it has at
most $\rho$ edges.

\subsection{Cardinal arithmetic}

\begin{lemma}[Small unions]
\label{lem:small-union}
Let $\rho$ be an infinite cardinal, let $\mu<\cf(\rho)$, and let
$(X_i:i\in I)$ satisfy $|I|\le\mu$ and $|X_i|<\rho$ for every $i\in I$.
Then
\[
 \left|\bigcup_{i\in I}X_i\right|<\rho.
\]
\end{lemma}

\begin{proof}
The set of cardinals $\{|X_i|:i\in I\}$ has size at most
$\mu<\cf(\rho)$, so
\[
 \sigma=\sup_{i\in I}|X_i|<\rho.
\]
Consequently
\[
 \left|\bigcup_{i\in I}X_i\right|
 \le |I|\cdot\sigma<\rho,
\]
because both factors are strictly below the infinite cardinal $\rho$.
\end{proof}

\begin{lemma}
\label{lem:cardinal-facts}
Let $\mu$ be an infinite cardinal and put
\[
 \rho=2^\mu,\qquad \kappa=\mu^+,\qquad R=\rho^+,\qquad
 \Lambda=2^\rho.
\]
Then:
\begin{enumerate}[label=\textup{(\roman*)}]
\item $\cf(\rho)>\mu$;
\item $\kappa\le\rho$;
\item $\Lambda^\rho=\Lambda$;
\item $R\le\Lambda$;
\item every subset of $R$ of cardinality at most $\rho$ is bounded in
$R$.
\end{enumerate}
\end{lemma}

\begin{proof}
Part~\textup{(i)} is K\"onig's theorem in the form
$\cf(2^\mu)>\mu$.  Cantor's theorem gives
$2^\mu\ge\mu^+$, proving~\textup{(ii)}.  For~\textup{(iii)},
\[
 (2^\rho)^\rho=2^{\rho\cdot\rho}=2^\rho.
\]
Part~\textup{(iv)} again follows from Cantor's theorem.  Finally,
$R=\rho^+$ is regular, so every subset of cardinality less than $R$ is
bounded.
\end{proof}

\begin{lemma}
\label{lem:cofinal-fibres}
If $\kappa\le\rho$ are infinite cardinals and $R=\rho^+$, then there is
a map $q:R\to\kappa$ every fibre of which is cofinal in $R$.
\end{lemma}

\begin{proof}
Regard $\kappa$ as its initial ordinal.  By ordinal division, every
$\alpha<R$ is uniquely of the form $\kappa\beta+\xi$ with
$\xi<\kappa$.  Set $q(\alpha)=\xi$.  For fixed $\xi$, the ordinals
$\kappa\beta+\xi$ $(\beta<R)$ all lie below $R$, because their
cardinalities are less than the initial ordinal $R$, and form a strictly
increasing set of cardinality $R$.  Since $R$ is regular, this set is
cofinal in $R$.
\end{proof}

\begin{lemma}
\label{lem:successors-cofinal}
If $\kappa$ is an uncountable limit cardinal, then the uncountable
successor cardinals below $\kappa$ are cofinal in $\kappa$.
\end{lemma}

\begin{proof}
Let $\alpha<\kappa$ be an ordinal and put
\[
  \nu=\max\{\aleph_0,|\alpha|\}.
\]
Because $\kappa$ is an initial ordinal and $\alpha<\kappa$, we have
$|\alpha|<\kappa$, and hence $\nu<\kappa$.  Since $\kappa$ is a limit
cardinal, it cannot equal the successor cardinal $\nu^+$; therefore
$\nu^+<\kappa$.  Moreover
\[
  \alpha<|\alpha|^+\le\nu^+.
\]
The cardinal $\nu^+$ is uncountable.  Thus every ordinal below $\kappa$
is bounded by an uncountable successor cardinal still below $\kappa$,
which proves cofinality.
\end{proof}

\subsection{The successor construction}

Fix an infinite cardinal $\mu$ and put
\[
 \kappa=\mu^+,
 \qquad
 \rho=2^\mu,
 \qquad
 R=\rho^+,
 \qquad
 \Lambda=2^\rho.
\]
The parameter ladder used throughout the construction is
\begin{equation}
 \mu<\kappa=\mu^+\le\rho=2^\mu<R=\rho^+\le\Lambda=2^\rho,
 \qquad \cf(\rho)>\mu.
 \label{eq:parameter-ladder}
\end{equation}
The cardinal
\[
 \delta(\rho)=\min\{\delta:\rho^\delta>\rho\}
\]
exists, since $\rho^\rho\ge2^\rho>\rho$.  For every nonzero
$\theta\le\mu$,
\[
 \rho^\theta=(2^\mu)^\theta=2^{\mu\theta}=2^\mu=\rho.
\]
Consequently $\delta(\rho)>\mu$ and hence
$\kappa=\mu^+\le\delta(\rho)$.  By
\Cref{thm:EGH-P,lem:P-monotone}, fix a graph
$S=GS_2(\rho)$, a label set $I$ of cardinality $\rho$, and a labelling
\[
 \ell:E(S)\longrightarrow I
\]
with property $P(S,I,\kappa)$.  We have
$|V(S)|=\rho$ and $|E(S)|\le\rho$.  Fix also a map
$q:R\to\kappa$ with every fibre cofinal in $R$, as supplied by
\Cref{lem:cofinal-fibres}.

For every $\alpha<R$, let $V_\alpha$ be a set of cardinality $\Lambda$,
with the sets $V_\alpha$ pairwise disjoint, and write
\[
 V_{<\alpha}=\bigcup_{\beta<\alpha}V_\beta,
 \qquad
 \rk(x)=\alpha\quad(x\in V_\alpha).
\]
An \emph{admissible reservoir at level $\alpha$} is an $I$-indexed
family $\mathbf X=(X_i:i\in I)$ satisfying
\begin{enumerate}[label=\textup{(R\arabic*)}]
\item $X_i\subseteq V_{<\alpha}$ for every $i\in I$;
\item every $X_i$ is empty or has cardinality $\rho$;
\item the nonempty $X_i$ are pairwise disjoint;
\item if $x\in\bigcup_{i\in I}X_i$, then
$q(\rk(x))\ne q(\alpha)$.
\end{enumerate}
Let $\mathcal R_\alpha$ denote the set of all admissible reservoirs at
level $\alpha$.

\begin{lemma}[Stage capacity]
\label{lem:stage-capacity}
For every $\alpha<R$,
\[
 |\mathcal R_\alpha|\le\Lambda.
\]
Moreover, the level $V_\alpha$ has enough points to contain, pairwise
vertex-disjointly, $\rho$ labelled copies of $(S,\ell)$ for every
reservoir in $\mathcal R_\alpha$.
\end{lemma}

\begin{proof}
Since $\alpha<R\le\Lambda$ and every earlier level has cardinality
$\Lambda$,
\[
 |V_{<\alpha}|\le |\alpha|\cdot\Lambda\le R\Lambda=\Lambda.
\]
Therefore
\[
 |[V_{<\alpha}]^\rho|\le\Lambda^\rho=\Lambda
\]
by \Cref{lem:cardinal-facts}.  For each coordinate $i\in I$, there are at
most $\Lambda$ choices for $X_i$, including the empty choice.  Since
$|I|=\rho$,
\[
 |\mathcal R_\alpha|
 \le(\Lambda+1)^\rho
 =\Lambda^\rho
 =\Lambda.
\]
For each pair $(\mathbf X,\eta)$ with
$\mathbf X\in\mathcal R_\alpha$ and $\eta<\rho$, a copy of $S$ uses
$\rho$ vertices.  Hence the total number of required vertices is at most
\[
 |\mathcal R_\alpha|\cdot\rho\cdot|V(S)|
 \le\Lambda\cdot\rho\cdot\rho
 =\Lambda.
\]
Thus all required copies fit pairwise disjointly inside $V_\alpha$.
\end{proof}

\begin{lemma}[Recursive construction]
\label{lem:calibration-construction}
There is a triple system $L_\kappa$ on
$\bigcup_{\alpha<R}V_\alpha$ with the following data fixed by a
transfinite recursion on $\alpha<R$.
For every $\alpha<R$, every $\mathbf X\in\mathcal R_\alpha$, and every
$\eta<\rho$, there is a labelled copy
$B_{\mathbf X,\eta}\subseteq V_\alpha$ of $(S,\ell)$, and the copies
installed at one level are pairwise vertex-disjoint.  For every installed
copy $B=B_{\mathbf X,\eta}$ there is an injection
\[
 \phi_B:
 \{e\in E(B):X_{\ell(e)}\ne\varnothing\}
 \longrightarrow \bigcup_{i\in I}X_i
\]
such that $\phi_B(e)\in X_{\ell(e)}$.  The edges of $L_\kappa$ are
exactly the triples
\[
 e\cup\{\phi_B(e)\}
\]
for installed copies $B$ and edges $e$ in the domain of $\phi_B$.
All reservoirs, copies, transported labels, and maps $\phi_B$ are fixed
independently of any later vertex colouring.
\end{lemma}

\begin{proof}
Proceed by transfinite recursion on $\alpha<R$.  At stage $\alpha$, all
levels below $\alpha$ and all data installed there have already been
fixed.  The family $\mathcal R_\alpha$ is therefore determined solely by
that earlier data.  By \Cref{lem:stage-capacity}, allocate in
$V_\alpha$, for every pair
$(\mathbf X,\eta)\in\mathcal R_\alpha\times\rho$, a copy
$B_{\mathbf X,\eta}$ of $S$, with all these copies pairwise disjoint.
Fix a bijection from $V(S)$ to each copy, transport adjacency along it,
and label every transported edge by the label of its preimage under
$\ell$.  Thus the entire labelled graph, not merely its underlying graph,
is transported.

Fix one installed copy $B=B_{\mathbf X,\eta}$.  Call $i\in I$
\emph{active} when $X_i\ne\varnothing$.  For each active $i$, the set of
edges of $B$ carrying label $i$ has cardinality at most
$|E(B)|\le\rho=|X_i|$, so choose an injection of that edge set into
$X_i$.  Since the active sets $X_i$ are pairwise disjoint, the union of
these labelwise injections is one injection $\phi_B$ with the stated
properties.  Declare $e\cup\{\phi_B(e)\}$ to be a triple whenever
$e\in\dom(\phi_B)$.  This completes stage $\alpha$ and hence the
recursion.

Points of $V_\alpha$ not allocated to an installed copy are not used as
base vertices at level $\alpha$; they may subsequently occur as apex
vertices of edges based at higher levels.  Every choice just described is
part of the recursion and is made before any colouring is considered.
\end{proof}

\begin{lemma}[Linearity]
\label{lem:calibration-linearity}
The triple system $L_\kappa$ supplied by
\Cref{lem:calibration-construction} is linear.
\end{lemma}

\begin{proof}
Consider two distinct triples.  If they are based in the same installed
copy $B\subseteq V_\alpha$, then their two graph bases are distinct edges
of a simple graph and hence share at most one base vertex.  Their apexes
are distinct because $\phi_B$ is injective, and an apex lies below level
$\alpha$, so it cannot equal any base vertex at level $\alpha$.  Thus the
triples meet in at most one vertex.

If the triples are based in different copies at the same level, their
bases are disjoint.  Each has only one apex, so they can meet only when
those apexes coincide, again in at most one vertex.

Finally, suppose their base levels are $\alpha<\beta$.  Every vertex of
the first triple has rank below $\beta$, whereas the second triple has
exactly one vertex below $\beta$, namely its apex.  Hence their
intersection is contained in that one apex.  These cases prove
linearity.
\end{proof}

\begin{lemma}[Canonical upper colouring]
\label{lem:calibration-upper}
The map
\[
 x\longmapsto q(\rk(x))
\]
is a proper $\kappa$-colouring of $L_\kappa$.  Consequently
$\chi(L_\kappa)\le\kappa$.
\end{lemma}

\begin{proof}
Let a triple be based at level $\alpha$.  Its two base vertices have
colour $q(\alpha)$.  Its apex belongs to an admissible reservoir at level
$\alpha$, so condition~\textup{(R4)} gives it a colour different from
$q(\alpha)$.  The triple is therefore not monochromatic.
\end{proof}

\begin{lemma}[Reservoir capture]
\label{lem:reservoir-capture}
Let $c:V(L_\kappa)\to\theta$ with $\theta<\kappa$.  Since
$V(L_\kappa)\ne\varnothing$, necessarily $\theta>0$.  Then there exist:
\begin{enumerate}[label=\textup{(\roman*)}]
\item a value $\xi<\kappa$;
\item an admissible reservoir $\mathbf X=(X_i:i\in I)$ at some level
$\alpha\in q^{-1}\{\xi\}$;
\item an injection $a\mapsto i_a$ from the set of large colour classes
$C_a=c^{-1}\{a\}$, where ``large'' means $|C_a|\ge\rho$, into $I$;
\item an installed copy $B=B_{\mathbf X,\eta}\subseteq V_\alpha$
for which every colour occurring on $B$ is large;
\end{enumerate}
such that $X_{i_a}\subseteq C_a$ for every large colour $a$.
The proof uses, explicitly,
\[
 |D|<\rho,
 \qquad
 \left|\{\rk(x):x\in\bigcup_iX_i\}\right|\le\rho<R,
 \qquad
 q^{-1}\{\xi\}\text{ is cofinal in }R,
\]
where $D$ is the union of the nonlarge colour classes.
\end{lemma}

\begin{proof}
Because $\theta<\kappa=\mu^+$ and cardinals are identified with their
initial ordinals, $|\theta|\le\mu$.  Put $C_a=c^{-1}\{a\}$ for
$a<\theta$, call $a$ large when $|C_a|\ge\rho$, and let
\[
 D=\bigcup\{C_a:|C_a|<\rho\}.
\]
There are at most $\mu$ nonlarge colour classes, each of cardinality less
than $\rho$.  Since $\cf(\rho)>\mu$, the small-union lemma gives
\begin{equation}
 |D|<\rho.
 \label{eq:D-small}
\end{equation}
At least one colour is large: otherwise $D=V(L_\kappa)$, contradicting
$|D|<\rho\le|V_0|$.

For $\xi<\kappa$, put
\[
 Q_\xi=\bigcup\{V_\beta:q(\beta)=\xi\}.
\]
For a fixed large colour $a$, at most one $\xi<\kappa$ can satisfy
\begin{equation}
 |C_a\setminus Q_\xi|<\rho.
 \label{eq:exceptional}
\end{equation}
Indeed, if distinct $\xi,\zeta$ both satisfied
\eqref{eq:exceptional}, then
$C_a\cap Q_\xi\subseteq C_a\setminus Q_\zeta$, so both
$C_a\cap Q_\xi$ and $C_a\setminus Q_\xi$ would have cardinality less
than $\rho$.  Their union is $C_a$, contradicting that $a$ is large.

There are at most $\mu$ large colours, and each excludes at most one value
of $\xi$.  Since $\kappa=\mu^+$, choose $\xi<\kappa$ that is not
exceptional for any large colour.  Because the number of large colours is
at most $\mu<\kappa\le\rho=|I|$, choose an injection $a\mapsto i_a$
from the large colours into $I$.  For each large $a$, choose
\[
 X_{i_a}\in[C_a\setminus Q_\xi]^\rho,
\]
and put $X_i=\varnothing$ for all remaining labels.  The nonempty
$X_i$ are pairwise disjoint because they lie in distinct colour classes.
Their union has cardinality exactly $\rho$: it contains one set of size
$\rho$ and is the union of at most $\mu\le\rho$ such sets.

Consequently
\[
 \left|\{\rk(x):x\in\bigcup_iX_i\}\right|\le\rho<R.
\]
The successor cardinal $R=\rho^+$ is regular, so this set of ranks is
bounded in $R$.  Since $q^{-1}\{\xi\}$ is cofinal in $R$, choose
$\alpha\in q^{-1}\{\xi\}$ strictly above every one of these ranks.
Then $\mathbf X=(X_i:i\in I)$ is admissible at level $\alpha$: all its
points lie below $\alpha$, and every active point lies outside $Q_\xi$,
so its $q$-value differs from $q(\alpha)=\xi$.

The construction installs $\rho$ pairwise vertex-disjoint copies
$B_{\mathbf X,\eta}$, $\eta<\rho$, for this reservoir.  Fewer than
$\rho$ of them meet $D$.  Indeed, choosing one point of $D$ from every
meeting copy gives an injection from the family of meeting copies into
$D$, because the copies are pairwise disjoint; this argument does not
assume that $\rho$ is regular.  Choose an installed copy $B$ disjoint
from $D$.  Every colour occurring on $B$ is then large, and by
construction $X_{i_a}\subseteq C_a$ for every large $a$.
\end{proof}

\begin{lemma}[Lower bound]
\label{lem:calibration-lower}
No colouring of $L_\kappa$ with fewer than $\kappa$ colours is proper.
Consequently $\chi(L_\kappa)\ge\kappa$.
\end{lemma}

\begin{proof}
Let $c:V(L_\kappa)\to\theta$ with $\theta<\kappa$.  Apply
\Cref{lem:reservoir-capture} and let $B=B_{\mathbf X,\eta}$ be the
captured labelled copy.  The transported labelling on $B$ has property
$P(S,I,\kappa)$.  Therefore there is one colour $a<\theta$ such that, for
every label $i\in I$, the colour class $c^{-1}\{a\}\cap B$ contains an
edge of label $i$.  By the conclusion of
\Cref{lem:reservoir-capture}, every colour occurring on $B$ is large.
Thus $a$ is large and $i_a$ is defined.  Choose a monochromatic
edge $e\subseteq B$ of label $i_a$.  This label is active, so
$e\in\dom(\phi_B)$, and the construction contains the triple
\[
 e\cup\{\phi_B(e)\}.
\]
Its apex lies in $X_{i_a}\subseteq C_a$, while both vertices of $e$ have
colour $a$.  The triple is monochromatic.  Thus $c$ is not proper.
\end{proof}

\begin{theorem}[Exact successor calibration]
\label{thm:successor-linear}
Let $\kappa=\mu^+$, where $\mu$ is infinite.  There is a linear triple
system $L_\kappa$ such that
\[
 \chi(L_\kappa)=\kappa
 \qquad\text{and}\qquad
 |V(L_\kappa)|\le 2^{2^\mu}.
\]
\end{theorem}

\begin{proof}
Take the system constructed in \Cref{lem:calibration-construction}.
It is linear by \Cref{lem:calibration-linearity}; its canonical colouring
and lower-bound lemma give
\[
 \kappa\le\chi(L_\kappa)\le\kappa.
\]
Finally,
\[
 |V(L_\kappa)|=R\Lambda=\Lambda=2^{2^\mu},
\]
because $R\le\Lambda$ and $\Lambda$ is infinite.
\end{proof}

\begin{remark}[Singular-cardinal audit]
\label{rem:singular-audit}
No regularity of $\kappa$ or $\rho=2^\mu$ is used in the successor
construction.  The only regular cardinal used in the lower bound is
$R=\rho^+$, and it is used solely to bound a set of at most $\rho$ ranks.
The estimate $|D|<\rho$ follows instead from
$\cf(\rho)>\mu$.  The selection of an installed copy disjoint from $D$
remains valid when $\rho$ is singular because the installed copies are
pairwise vertex-disjoint, so the copies meeting $D$ inject into $D$.
The lift argument elsewhere in the paper similarly uses no regularity of
its target cardinal.
\end{remark}

\begin{corollary}
\label{cor:all-linear}
For every uncountable cardinal $\kappa$, there is a linear triple system
of chromatic number exactly $\kappa$.
\end{corollary}

\begin{proof}
The successor case is \Cref{thm:successor-linear}.  If $\kappa$ is an
uncountable limit cardinal, use \Cref{lem:successors-cofinal} to choose
uncountable successor cardinals $(\kappa_i)_{i\in J}$ cofinal in
$\kappa$, and take the disjoint union of the corresponding
$L_{\kappa_i}$.  A disjoint union of linear systems is linear, and its
chromatic number is the supremum of the chromatic numbers of its
components:
\[
 \chi\left(\coprod_{i\in J}L_{\kappa_i}\right)
 =\sup_{i\in J}\kappa_i=\kappa.
\]
\end{proof}

\section{Exact spectra and Erd\H{o}s Problem~\#1177}
\label{sec:spectrum}

\begin{corollary}[Exact-spectrum class dichotomy]
\label{thm:spectrum}
For every finite triple system $F$,
\[
\Spec(F)=
\begin{cases}
\varnothing,&F\in\Bclass,\\[2mm]
\{\lambda\in\Card:\lambda>\aleph_0\},&F\notin\Bclass.
\end{cases}
\]
\end{corollary}

\begin{proof}
Let $F^\circ$ be obtained from $F$ by deleting isolated vertices.  By
\Cref{lem:isolated-reduction}, $F$ and $F^\circ$ have the same spectrum
and the same membership status in $\Bclass$.  We may therefore assume
that $F$ has no isolated vertices.  If $E(F)=\varnothing$, then
$F\in\Bclass$ and every infinite host contains $F$, so
$\Spec(F)=\varnothing$ as asserted.  Hence assume $E(F)\ne\varnothing$.

If $F\in\Bclass$, then $F$ is obligatory by
\Cref{thm:classification}; hence no uncountably chromatic $F$-free
system exists.

Suppose $F\notin\Bclass$, and let $\lambda$ be any uncountable
cardinal.  Apply the obstruction trichotomy, \Cref{cor:obstruction-trichotomy}.

In case~\textup{(i)}, the exact-$\lambda$ linear system from
\Cref{cor:all-linear} omits $F$.

In case~\textup{(ii)}, $F$ has no bridge selector.  Consequently
$\Lift(K_\lambda,\lambda)$ is $F$-free by
\Cref{thm:bridge-trace,rem:bridge-trace-isolated}, and it has chromatic
number exactly $\lambda$ by \Cref{thm:lift-chromatic}.

In case~\textup{(iii)}, let $m$ be the length of an odd Berge cycle of
$F$ and put $s=(m-1)/2$.  Choose, by \Cref{thm:EH-odd-girth}, an
exact-$\lambda$ graph $A$ containing no odd cycle of length at most $m$,
in particular no $C_m$.  By \Cref{lem:cycle-selector}, every bridge
selector has a derivative containing an actual $C_m$.  Hence
\Cref{thm:bridge-trace,rem:bridge-trace-isolated} shows that
$\Lift(A,\lambda)$ omits $F$, and \Cref{thm:lift-chromatic} gives exact
chromatic number $\lambda$.
Thus every uncountable $\lambda$ belongs to $\Spec(F)$.
\end{proof}

\begin{corollary}[Erd\H{o}s Problem~\#1177]
\label{thm:1177}
The three assertions in the current formulation of Erd\H{o}s
Problem~\#1177 have truth values
\[
 \textnormal{yes},\qquad \textnormal{no},\qquad \textnormal{yes}.
\]
\end{corollary}

\begin{proof}
\emph{Part~\textup{(1)}.}
Assume $F_G(\aleph_1)\ne\varnothing$.  Replace $G$ by $G^\circ$ if
necessary.  By \Cref{lem:isolated-reduction}, this neither changes the
avoidance class in any infinite host nor changes whether $G$ is
obligatory.  Since an exact-$\aleph_1$ $G$-free system exists, $G$ is not
obligatory, and hence $G\notin\Bclass$ by \Cref{thm:classification}.
In particular $E(G)\ne\varnothing$, because a finite edgeless system
embeds in every infinite host.  Apply
\Cref{cor:obstruction-trichotomy} to $G$.

In case~\textup{(i)}, use $L_{\aleph_1}$ from
\Cref{thm:successor-linear} with $\mu=\aleph_0$.  It is $G$-free, has
chromatic number $\aleph_1$, and has at most
$2^{2^{\aleph_0}}$ vertices.

In case~\textup{(ii)}, use $\Lift(K_{\aleph_1},\aleph_1)$.  In
case~\textup{(iii)}, use $\Lift(A,\aleph_1)$, where $A$ is an
exact-$\aleph_1$ graph on $\aleph_1$ vertices avoiding the corresponding
odd cycle.  In either linear case, $|V(A)|=\aleph_1$ and
$|E(A)|\le\aleph_1$.  Therefore
\[
 |T(A,\aleph_1)|
 \le (\aleph_1)^{<\aleph_1}
 = (\aleph_1)^{\aleph_0}
 = 2^{\aleph_0}.
\]
The equality is the explicit ZFC calculation
\[
 2^{\aleph_0}
 \le (\aleph_1)^{\aleph_0}
 \le (2^{\aleph_0})^{\aleph_0}
 =2^{\aleph_0}.
\]
The first inequality uses $2\le\aleph_1$, and the second uses
$\aleph_1\le2^{\aleph_0}$.  Finally,
$2^{\aleph_0}\cdot\aleph_1=2^{\aleph_0}$, so multiplication by
$|V(A)|=\aleph_1$ does not increase the cardinal.  Thus every linear
case has a witness of size at most $2^{\aleph_0}$, and the asserted
$2^{2^{\aleph_0}}$ bound follows.

\emph{Part~\textup{(2)}.}
Let
\[
 T_0=\bigl\{\{a,b,c\},\{a,b,d\}\bigr\}.
\]
Because our hypergraphs are simple, a triple system is $T_0$-free
exactly when it is linear: two distinct triples violate linearity
precisely when they share a pair.
Thus $L_{\aleph_1}$ witnesses $F_{T_0}(\aleph_1)\ne\varnothing$.

Let $C_7^{(3)}$ be the private-vertex expansion of the graph cycle
$C_7$; with indices modulo $7$, its edges are
\[
 \{x_i,x_{i+1},y_i\},
\]
where all displayed vertices are distinct.  Choose an exact-$\aleph_1$
graph $A$ with no $C_7$.  In the Levi graph of $C_7^{(3)}$, the two incidences to the core
vertices lie on the displayed Levi cycle, whereas the incidence to
$y_i$ is a bridge.  Hence the private incidences form the unique bridge
selector.  Equivalently, \Cref{lem:cycle-selector} shows that its unique
nontrivial derivative is the actual graph cycle $C_7$.  By
\Cref{thm:bridge-trace}, $\Lift(A,\aleph_1)$ is
$C_7^{(3)}$-free, and by \Cref{thm:lift-chromatic} it has chromatic
number $\aleph_1$.  Thus $F_{C_7^{(3)}}(\aleph_1)\ne\varnothing$.

Hajnal and Komj\'ath proved that $C_n^{(3)}$ is linearly obligatory for
$n\notin\{2,3,5\}$ \cite{HK2008}; see the precise restatement in
\cite[\S3.7, p.~43]{ReiherGirth}.  In particular, every uncountably
chromatic linear triple system contains $C_7^{(3)}$.  A common member of
$F_{T_0}(\aleph_1)$ and $F_{C_7^{(3)}}(\aleph_1)$ would be an
uncountably chromatic linear system omitting $C_7^{(3)}$, a
contradiction.  Their intersection is empty.

\emph{Part~\textup{(3)}.}
If $F_G(\kappa)\ne\varnothing$ for one uncountable $\kappa$, then $G$
is not obligatory, equivalently $G\notin\Bclass$.  By
\Cref{thm:spectrum}, $\Spec(G)$ contains every uncountable cardinal.
Hence $F_G(\lambda)\ne\varnothing$ for every uncountable $\lambda$.
\end{proof}

\section{Concluding perspective}

The bridge-trace theorem and the finite running-intersection argument
prove the classification of obligatory triple systems, resolving
Erd\H{o}s Problem~\#593.  Exact linear calibration is a separate
exact-cardinal result.  Combined with complete-rank lifts over
high-odd-girth graphs, it yields the all-or-nothing exact spectrum
theorem and the stated answers to Erd\H{o}s Problem~\#1177.

\appendix

\section{Exact interfaces to the imported theorems}
\label{app:external-interface}

This appendix records the source notation and the precise parameter
translations used in the body of the paper.  Every imported result listed
below is a theorem of ZFC.  None of the sources is used to infer more than
is stated here.

\subsection{The Erd\H{o}s--Hajnal--Rothschild obstruction}

In the notation of Erd\H{o}s, Hajnal, and Rothschild, the expression
\[
 R(\alpha,\beta,\gamma,k,i)
\]
denotes their partition property for $k$-uniform set systems with
intersection threshold $i$.  Their Theorem~2 on p.~532 states that, if
\[
 \beta=\omega_\xi,\qquad 3\le k<\omega,\qquad 2\le i\le k-1,
 \qquad
 \alpha=(\exp_{k-i}(\beta))^+,
\]
then $R(\alpha,\beta,2,k,i)$ fails
\cite[Theorem~2, p.~532]{EHR1973}.  We use only the standard specialization
$k=3$ and $i=2$: it yields an uncountably chromatic linear triple system,
and therefore every finite triple system containing two distinct edges
with at least two common vertices is non-obligatory.  This input supplies
an uncountable lower bound, not a prescribed exact chromatic cardinal; no
exactness is required in the proof of Problem~\#593.

\subsection{Exact high odd girth}

In the notation quoted as Theorem~C by Erd\H{o}s--Galvin--Hajnal, the
Erd\H{o}s--Hajnal theorem says that for every cardinal $\kappa>\aleph_0$
and every positive integer $i$ there is a graph $A$ such that
\[
 |V(A)|=\chi(A)=\kappa
 \quad\text{and}\quad
 C_{2j+1}\not\hookrightarrow A\quad(0<j<i).
\]
See \cite[Theorem~7.4, p.~76]{EH1966} and the verbatim restatement
\cite[Theorem~C, p.~428]{EGH}.  Substituting $i=s+1$ excludes every
ordinary, not necessarily induced, odd cycle of length at most $2s+1$.
This theorem gives equality $|V(A)|=\chi(A)=\kappa$, and it is used at the
same target cardinal as the lift.

\subsection{The simultaneous EGH labelling}

Definition~6.2 on p.~448 of Erd\H{o}s--Galvin--Hajnal defines
$P(S,I,\delta)$ by the quantifier order
\begin{equation}
 \boxed{
 \forall c\ \exists a\ \forall i\ \exists e
 }
 \label{eq:EGH-quantifier-order}
\end{equation}
with the meanings made explicit as follows: for every $\theta<\delta$ and
every colouring $c:V(S)\to\theta$, there is one colour $a<\theta$ such
that, for every label $i\in I$, some edge $e$ of label $i$ is contained in
$c^{-1}\{a\}$.  In particular, the colour $a$ is common to all labels; the
weaker order $\forall i\,\exists a_i\,\exists e$ is not what is used.

Their Corollary~9.7 on p.~461 states, in the original parameters, that if
$\rho$ is infinite,
\[
 \delta(\rho)=\min\{\delta:\rho^\delta>\rho\},
\]
and $1<n<\omega$, then
\[
 P(GS_n(\rho),\rho,\delta(\rho))
\]
holds \cite[Definition~6.2, p.~448 and Corollary~9.7, p.~461]{EGH}.  We
substitute $n=2$, transport the label set $\rho$ along a bijection onto
$I$, and then use monotonicity from $\delta(\rho)$ down to
$\kappa=\mu^+$.  The imported theorem supplies the labelled graph and the
simultaneous lower-colour property; it does not by itself assert the
chromatic number of the triple system constructed in
\Cref{sec:linear}.  Exactness there is proved separately by
\Cref{lem:calibration-upper,lem:calibration-lower}.

\subsection{Reiher's obligatory expansions}

Reiher's Theorem~1.2 states that the $k$-uniform private expansion of
$K_{n,n}$ is obligatory for every finite $n$ and every finite uniformity
$k\ge2$ \cite[Theorem~1.2]{ReiherObligatory}.  We substitute $k=3$.
Obligatoriness means containment in every host of uncountable chromatic
number; it is not an exact-cardinal assertion.  Passing to finite
subhypergraphs then gives obligatoriness of $J^+$ for every finite
bipartite graph $J$.

\subsection{The linearly obligatory loose seven-cycle}

A finite triple system is \emph{linearly obligatory} when it occurs in
every linear triple system of uncountable chromatic number.  Hajnal and
Komj\'ath prove that the loose/private-vertex cycle $C_n^{(3)}$ is linearly
obligatory for $n\notin\{2,3,5\}$ \cite{HK2008}.  The precise statement
and the same cycle convention are recorded in Reiher's survey
\cite[\S3.7, p.~43]{ReiherGirth}.  We substitute $n=7$.  This result is
not an exact-cardinal theorem, but it applies in particular to every
exact-$\aleph_1$-chromatic linear triple system, which is exactly the use
made in Problem~\#1177(2).

\section*{Acknowledgements}

The author acknowledges the use of OpenAI's ChatGPT during the
preparation of this manuscript.  While it was used for ideation,
formulation, proof exploration and refinement, narrowing the search
space, programming, LaTeX formatting and other forms of orchestration,
the author nonetheless takes full responsibility for the accuracy of the
final contents of this paper.

The accompanying Lean formalization \cite{LeanFormalization} was developed
by the author with Aristotle (Harmonic); the author likewise takes full responsibility
for the formal statements' fidelity to the results of this paper.


\begin{thebibliography}{99}

\bibitem{deBruijnErdos}
N.~G. de Bruijn and P.~Erd\H{o}s,
\emph{A colour problem for infinite graphs and a problem in the theory of relations},
Nederl. Akad. Wetensch. Proc. Ser. A \textbf{54} = Indag. Math. \textbf{13}
(1951), 369--373.

\bibitem{Erdos1995}
P.~Erd\H{o}s,
\emph{On some problems in combinatorial set theory},
Publ. Inst. Math. (Beograd) (N.S.) \textbf{57(71)} (1995), 61--65.

\bibitem{Erdos593}
T.~F. Bloom,
\emph{Erd\H{o}s Problem \#593},
Erd\H{o}s Problems,
\url{https://www.erdosproblems.com/593},
accessed 23 June 2026.

\bibitem{EHR1973}
P.~Erd\H{o}s, A.~Hajnal, and B.~L. Rothschild,
\emph{On chromatic number of graphs and set-systems},
in \emph{Cambridge Summer School in Mathematical Logic}
(Cambridge, 1971), Lecture Notes in Mathematics, vol.~337,
Springer, Berlin--New York, 1973, pp.~531--538.

\bibitem{Va99}
B.~Bollob\'as et al. (collectors),
\emph{Some of Paul's Favorite Problems},
booklet circulated at the conference \emph{Paul Erd\H{o}s and his
Mathematics}, Budapest, July 1999, Problem~7.94, p.~14,
\url{https://web.math.pmf.unizg.hr/~vjekovac/EP/Some_of_Pauls_favorite_problems.pdf}.

\bibitem{Erdos1177}
T.~F. Bloom,
\emph{Erd\H{o}s Problem \#1177},
Erd\H{o}s Problems,
\url{https://www.erdosproblems.com/1177},
accessed 23 June 2026.

\bibitem{EGH}
P.~Erd\H{o}s, F.~Galvin, and A.~Hajnal,
\emph{On set-systems having large chromatic number and not containing
prescribed subsystems},
in \emph{Infinite and Finite Sets} (Keszthely, 1973),
Colloq. Math. Soc. J\'anos Bolyai, vol.~10, North-Holland, 1975,
pp.~425--513.

\bibitem{EH1966}
P.~Erd\H{o}s and A.~Hajnal,
\emph{On chromatic number of graphs and set-systems},
Acta Math. Acad. Sci. Hungar. \textbf{17} (1966), 61--99,
\href{https://doi.org/10.1007/BF02020444}{doi:10.1007/BF02020444}.

\bibitem{HK2008}
A.~Hajnal and P.~Komj\'ath,
\emph{Obligatory subsystems of triple systems},
Acta Math. Hungar. \textbf{119} (2008), no.~1--2, 1--13,
\href{https://doi.org/10.1007/s10474-007-6231-2}{doi:10.1007/s10474-007-6231-2}.

\bibitem{Komjath2008}
P.~Komj\'ath,
\emph{An uncountably chromatic triple system},
Acta Math. Hungar. \textbf{121} (2008), no.~1--2, 79--92,
\href{https://doi.org/10.1007/s10474-008-7179-6}{doi:10.1007/s10474-008-7179-6}.

\bibitem{Komjath2001}
P.~Komj\'ath,
\emph{Some remarks on obligatory subsystems of uncountably chromatic
triple systems},
Combinatorica \textbf{21} (2001), no.~2, 233--238,
\href{https://doi.org/10.1007/s004930100021}{doi:10.1007/s004930100021}.

\bibitem{WangDuanGerbnerKarim}
Y.~Wang, M.~Duan, D.~Gerbner, and H.~Hama Karim,
\emph{On the largest chromatic number of $F$-free hypergraphs},
preprint, 2026,
\href{https://arxiv.org/abs/2604.21551}{arXiv:2604.21551}.

\bibitem{ReiherObligatory}
C.~Reiher,
\emph{Obligatory hypergraphs},
Proc. Amer. Math. Soc., to appear,
\href{https://doi.org/10.1090/proc/17021}{doi:10.1090/proc/17021}.

\bibitem{ReiherGirth}
C.~Reiher,
\emph{Graphs of large girth},
preprint, 2024,
\href{https://arxiv.org/abs/2403.13571}{arXiv:2403.13571}.

\bibitem{LeanFormalization}
E.~Li,
\emph{Lean formalization of the resolutions of Erd\H{o}s Problems 593 and 1177},
GitHub repository,
\url{https://github.com/ericlisg/erdos-593-1177-lean},
2026.

\end{thebibliography}
\end{document}